\documentclass[a4paper,11pt]{article}
\usepackage{amsmath,mathtools, xfrac}
\usepackage{amsfonts}
\usepackage{amssymb}
\usepackage{mathrsfs}
\usepackage{hyperref}
\usepackage[retainorgcmds]{IEEEtrantools}
\usepackage{enumerate}
\usepackage{enumitem}
\usepackage{amsthm}
\usepackage{fullpage}
\usepackage{graphicx}
\usepackage{caption}
\usepackage{color}
\usepackage{float}
\usepackage{subfigure}
\usepackage[final]{microtype}
\usepackage{verbatim}
\usepackage{qtree}
\usepackage{tikz}
\usepackage{tikz-qtree}
\usetikzlibrary{arrows,automata,positioning}

\theoremstyle{definition} \newtheorem{Definition}{Definition}[section]
\theoremstyle{plain} \newtheorem{Theorem}[Definition]{Theorem}
\theoremstyle{plain} \newtheorem{corollary}[Definition]{Corollary}
\theoremstyle{plain} \newtheorem{lemma}[Definition]{Lemma}
\theoremstyle{definition} \newtheorem{Remark}[Definition]{Remark}
\theoremstyle{plain}
\newtheorem{proposition}[Definition]{Propostion}
\theoremstyle{plain} \newtheorem{claim}[Definition]{Claim}
\theoremstyle{plain} 
\theoremstyle{definition}
\newtheorem{example}[Definition]{Example}
\theoremstyle{definition}
\newtheorem{construction}[Definition]{Construction}
\theoremstyle{plain}

\newcommand{\gen}[1]{\left\langle #1 \right\rangle}

\newcommand{\R}{\mathcal{R}}

\newcommand{\ew}{\varepsilon}

\newcommand{\im}{\mbox{ im}}

\newcommand{\id}{\mathop{\mathrm{id}}}

\newcommand{\aut}[1]{\mathop{\mathrm{Aut}}({#1})}

\newcommand{\out}[1]{\mathop{\mathrm{Out}}({#1})}

\newcommand{\wn}{W_{n}}
\newcommand{\wns}{\wn^{\ast}}
\newcommand{\wnl}[1]{\wn^{#1}}
\newcommand{\rwnl}[1]{\mathsf{W}_{n}^{#1}}

\newcommand{\T}[1]{\mathcal{#1}}

\newcommand{\CCn}{\mathfrak{C}_{n}}

\newcommand{\Tnr}{T_{n,r}}
\newcommand{\Gnr}{G_{n,r}}

\newcommand{\Bnr}{\mathcal{B}_{n,r}}
\newcommand{\TBnr}{\T{T}\Bnr}

\newcommand{\On}{\T{O}_{n}}
\newcommand{\SOn}{\T{SO}_{n}}
\newcommand{\TSOn}{\T{TSO}_{n}}
\newcommand{\TSLn}{\T{TSL}_{n}}
\newcommand{\Onr}{\T{O}_{n,r}}
\newcommand{\Ons}[1]{\T{O}_{n,#1}}
\newcommand{\Oms}[2]{\T{O}_{#1,#2}}
\newcommand{\XOn}{\T{X}_{n}}
\newcommand{\XOno}[1]{\T{X}_{\Omega{(#1)}}}

\newcommand{\TOn}[1]{\T{TO}_{#1}}
\newcommand{\TLn}[1]{\T{TL}_{#1}}
\newcommand{\SLn}[1]{\T{SL}_{#1}}
\newcommand{\TOnr}{\T{T}\Onr}
\newcommand{\TOns}[1]{\T{T}\T{O}_{n,#1}}
\newcommand{\TOms}[2]{\T{T}\T{O}_{#1,#2}}

\newcommand{\xt}{X_{2}}
\newcommand{\xtp}{\xt^{+}}
\newcommand{\xts}{\xt^{\ast}}
\newcommand{\Xn}{X_{n}}
\newcommand{\xn}{\Xn}
\newcommand{\Xns}{\Xn^{*}}
\newcommand{\xns}{\Xns}
\newcommand{\xnp}{\Xn^{+}}

\newcommand{\xno}{X_n^{\omega}}

\newcommand{\Gn}[1]{\mathfrak{G}(#1)}

\newcommand{\Ln}[1]{\mathcal{L}_{#1}}
\newcommand{\pn}[1]{\mathcal{P}_{#1}}
\newcommand{\spn}[1]{\widetilde{\T{P}}_{#1}}

\newcommand{\core}{\mathrm{Core}}

\newcommand{\Un}[1]{\mathfrak{U_{#1}}}
\newcommand{\ul}[1]{\underline{#1}}

\newcommand{\lcp}[1]{\mbox{lcp}({#1})}

\newcommand{\leqlex}{\le_{\mbox{lex}}}
\newcommand{\lelex}{<_{\mbox{lex}}}

\newcommand{\leslex}{<_{\mbox{slex}}}

\newcommand{\pd}[1]{\Omega(#1)}

\newcommand{\Z}{\mathbb{Z}}
\newcommand{\N}{\mathbb{N}}

\newcommand{\tran}{\mathop{\mathrm{Tr}}}

\newcommand{\simeqI}{\simeq_{{\bf{I}}}}
\newcommand{\rot}{\mathmbox{rot}}
\newcommand{\simrot}{\sim_{{\rot}}}
\newcommand{\rotclass}[1]{[#1]_{\simrot}}

\newcommand{\rsig}{\mathrm{\overline{sig}}}
\newcommand{\shift}[1]{\sigma_{#1}}

\makeatletter
\renewcommand*{\eqref}[1]{%
  \hyperref[{#1}]{\textup{\tagform@{\ref*{#1}}}}%
}
\makeatother

\begin{document}
\author{
  Olukoya Feyishayo,\\
  Institute Of Mathematics, \\
  University of Aberdeen,\\
  King's College, \\
  Fraser Noble Building,\\
  Aberdeen AB24 3UE\\
  Scotland\\
 \texttt{feyisayo.olukoya@abdn.ac.uk}
}
\title{An automata theoretic proof that \texorpdfstring{$\out{T}\cong \Z/ 2\Z$ }{Lg}  and some embedding results for \texorpdfstring{$\out{V}$}{Lg}}
\maketitle
\begin{abstract}
In a seminal paper, Brin demonstrates that the outerautomorphism group of Thompson group $T$ is  isomorphic to the cyclic group of order two. In this article, building on characterisation of automorphisms of the Higman-Thompson groups $\Gnr$ and $\Tnr$ as  groups of transducers, we give a new proof, automata theoretic in nature, of Brin's result.

We also demonstrate that the group of outerautomorphisms of Thompson's group $V = G_{2,1}$ contains an isomorphic copy of Thompson's group $F$. This extends a result  of the author  demonstrating that whenever $n \ge 3$ and $1 \le r < n$ the outerautomorphism groups of $G_{n,r}$ and $T_{n,r}$ contain an isomorphic copy of $F$.

\end{abstract}

\section{Introduction}

In the seminal paper  \cite{MBrin2}, Brin demonstrates that the outerautomorphism group of Thompson group $T$ is  isomorphic to the cyclic group of order two. In this article, building on characterisation of automorphisms of the Higman-Thompson groups $\Gnr$ and $\Tnr$ as  groups of transducers, we give a new proof, automata theoretic in nature, of Brin's result.

The Higman-Thompson groups $\Gnr$ and $\Tnr$ are significant groups in geometric and combinatorial group theory. For not only were they the first discovered family of finitely presented simple groups (the derived subgroup of $\Gnr $ and $\Tnr$ is simple when $n$ is even), they also crop up in numerous areas of group theory. In a follow up paper to \cite{MBrin2}, Brin and Guzman (\cite{MBrinFGuzman}), characterise automorphisms of the groups $T_{n, n-1}$. Their techniques however does not extend to the groups $G_{n,r}$ (for any $r$) or $\Tnr$ (when $r$ is no equal to $n-1$). This remained a challenge for the community, until the breakthrough paper \cite{BlkYMaisANav} characterising the automorphism group of $\Gnr$ as a group $\Bnr$ of transducers and the follow up paper \cite{OlukoyaAutTnr} extending this result to show that $\aut{\Tnr}$ is a subgroup $\TBnr$ of $\Bnr$.

An unexpected consequence, developed further in forthcoming articles \cite{BleakCameronOlukoya1,BleakCameronOlukoya2,BelkBleakCameronOlukoya3}, of the characterisations of the outerautomorphisms groups $\Onr$ of $\Gnr$ given in the paper \cite{BlkYMaisANav}, is a connection to the group $\aut{\xn^{\Z}, \shift{n}}$ of automorphisms of the shift dynamical system. The group $\aut{\xn^{\Z}, \shift{n}}$ is of course an important and well-studied group in symbolic dynamics. It turns out (\cite{BleakCameronOlukoya2}) that the group $\Ons{n-1}$ has a subgroup $\Ln{n}$  which sits in short exact sequence: 
$$1 \to \gen{\shift{n}}  \to \aut{\xn^{\Z}, \shift{n}} \to \Ln{n} \to 1$$ where $\gen{\shift{n}}$ is the centre of $\aut{\xn^{\Z}, \shift{n}}$ by a result in \cite{Ryan2}. The inclusion $ \Ons{1} \le \Onr \le  \Ons{n-1}$ of the outerautomorphism groups (\cite{BlkYMaisANav}), means that for each $r$ there is a corresponding group $\Ln{n,r} = \Ln{n} \cap \Onr$. We write $\On$ for the group $\Ons{n-1}$.

By results in \cite{OlukoyaAutTnr}, the group $\TOn{n,r}$ of outerautomorphisms of $\Tnr$ is a subgroup of $\Onr$ and we also have inclusions $\TOn{n, 1} \le \TOnr \le \Ons{n-1}$. In this context,  Brin's result in \cite{MBrin2} can be stated as follows:  $\TOn{2}$ consists only of elements in the intersection $\TOn{2} \cap \Ln{2}$ and this intersection contains a single non-trivial element. It is therefore natural to study the intersection $\TOn{n} \cap \Ln{n}$ for $n > 2$. By results in \cite{OlukoyaAutTnr} it is known that the set $\TOn{n} \backslash \Ln{n}$ is infinite, when $n >3$, in fact it contains an isomorphic copy of Thompson's group $F$. In this article we prove the following result:

\begin{Theorem}
		Let $n \ge 2$. Then the group  $\TOn{n} \cap \Ln{n}$ is isomorphic to the group $\Z/ 2\Z \times \Z^{\Omega(n-1)} \times \Z/l\Z$, where $n = p_1^{l_1} \ldots p_{\pd{n}}^{l_{\pd{n}}}$ is the prime decomposition of $n$ and $l = \gcd(l_1, \ldots, l_{\pd{n}})$.
\end{Theorem}

We note that when $n$ is prime, $\TOn{n} \cap \Ln{n} \cong \Z/2\Z$.

As mentioned above it is shown in the article \cite{OlukoyaAutTnr} that when $n>3$, $\TOnr$ and so $\Onr$ contain an isomorphic copy of Thompson's group  $F$. This is in fact an extension of a result in \cite{MBrinFGuzman} for the group $\TOns{n-1}$. It is of course impossible that $\TOn{2}$ contains a copy of $F$. The group $\mathcal{O}_{2}$ is however infinite (\cite{BlkYMaisANav}) and so it is natural to ask if  $\mathcal{O}_{2}$ contains a copy of $F$. This is  the case:

\begin{Theorem}\label{Thm:isocopyofF}
	Let $n \ge 2$ then $\Ons{r}$ contains an isomorphic copy of Thompson's group $F$.
\end{Theorem}

The proof of this result is worth saying a few words about as it explores the connections to symbolic dynamics.

 A standard technique for embedding various groups into the group $\aut{\xn^{\Z}, \shift{n}}$ is the use of \emph{marker automorphisms}. This appeared first in the significant paper of Hedlund \cite{Hedlund} and has been developed and generalised in several other papers (\cite{BoyleKrieger, BoyleLindRudolph88,VilleS18, KimRoush}). Perhaps the most significant use of this technique is in the paper \cite{KimRoush} where it is used to prove that the group $\aut{\xn^{\Z}, \shift{n}}$ contains an isomorphic copy of $\aut{X_{m}^{\Z}, \shift{m}}$ for any $m$. Our proof of  Theorem~\ref{Thm:isocopyofF} extends the marker technique to the groups $\On$. Specifically we prove using the marker construction that for each $n$ there is a subgroup $\On^{x}$ of $\On$ containing an isomorphic copy of $F$ which can be embedded in the group $\Oms{m}{1}$. It is still open whether the groups $\On$ are bi-embeddable. 

The paper is orgnanised as follows. In Section~\ref{Section:prelim} we introduce the various key definitions and background results we require; Section~\ref{Section:combiproof} contains the automata theoretic proof that $\TOn{2}$ is isomorphic to $\Z/2\Z$; Section~\ref{Section:intersection} studies the intersection of the group $\TOn{n}$ and $\Ln{n}$; the paper concludes in Section~\ref{Section:embeddingF} where we prove that $\On$ contains an isomorphic copy of $F$  for any $n \ge 2$.

 \section*{Acknowledgements}
The author wishes to acknowledge support from EPSRC research grant EP/R032866/1 and Leverhulme Trust Research Project Grant RPG-2017-159.

\section{Preliminaries}\label{Section:prelim}

This section contains the background definitions and results that we require in this article.

Throughout $\xn$ denotes the set $\{0,1,\ldots, n-1\}$.
\subsection{Words and Cantor space}
  Let $X$ be a finite set, we write $X^{\ast}$ for the set of all finite strings (including the empty string $\ew$) over the alphabet $X$. Write $X^{+}$ for the set $X^{\ast} \backslash \{\ew\}$. For an element $w \in X$, we write $|w|$ for the length of $w$. Thus, the empty string is the unique element of $X^{\ast}$ such that $|\ew| = 0$. For elements $x,y \in X^{\ast}$ we write $xy$ for the concatenation of $x$ and $y$.  Give $x \in X^{+}$ and $k \in \N$, we write $x^{k}$ for the word $w_1w_2\ldots w_{k}$ where $w_i = x$ for all $1 \le i \le k$. Given $k \in \N$ we write $X^{k}$ for the set of all words in $X^{+}$ of length exactly $k$.

Let $w \in X^{+}$ and write $w= w_1 w_2 \ldots w_{r}$ for $w_i \in X$, $1 \le i \le r$. Then for some $j \in \N$, $1 \le j \le r$ we say that $w' = w_{j}\ldots w_{r}w_1 \ldots w_{r}$ is the \emph{$j$'th rotation of $w$}. A word $w'$ is called a \emph{rotation} of $w$ if it is a $j$\textsuperscript{th} rotation of $w$ for some $1 \le j \le |w|$. If $j > 1$, then $w'$ is called a \emph{non-trivial rotation of} $w$. The word $w$ is said to be a prime word if $w$ is not equal to a non-trivial rotation of itself. Alternatively, $w$ is a prime word if there is no word $\gamma \in X^{+}$ such that $w = \gamma ^{|w|/|\gamma|}$. 

The relation on $X^{\ast}$ which relates two words if one is a  rotation of the other is an equivalence relation. We write $\simrot$ for this equivalence relation. It will be clear from the context which alphabet is meant.

We write $\wns$ for the set of all prime words over the alphabet $\xn$. For $k \in \N$, $\wnl{k}$ will denote the set of all prime words of length $k$ over the alphabet. We shall write $\rwnl{\ast}$ for the set of equivalence classes $\wns/\simrot$; $\rwnl{+}$  for the set of equivalence classes $\wnl{+}/\simrot$; $\rwnl{k}$ for the set of equivalence classes $\wnl{k}/\simrot$.

For $x,y \in X^{\ast}$ we write, $x \le y$ if $x$ is a prefix of $y$, in this case we write $y-x$ for the word $z \in X^{\ast}$ such that $xz = y$. The relation $\le$ is  a partial order on $X^{\ast}$.

We are require some additional orderings on the set $\xn^{\ast}$. Let $\xn$ be ordered according to the ordering $\le$ induced from $\N$. Then, for elements $x, y \in \xns$ we write $x \lelex y$ if there is a word $w \in \xns$ and $i \le j \in \xn$ such that  $wi$ is a prefix of $i$ and $wj$ is a prefix of $y$. We note that $\lelex$ is a partial order on $\xns$ called the \emph{lexicographic ordering}.

The lexicographic ordering can be extended to a total  order $\leslex$ on $\xns$ as follows. We write $x \leslex y$ if  either $x \lelex y$ or $|x| \le |y|$. The order $\leslex$ is called the \emph{short-lex ordering}.

An \emph{(positive) infinite sequence} over the alphabet $X$ is a map $x: \N \to X$ and we write such a sequence as $x_{0}x_1 x_2\ldots$ where $x_{i} = x(i)$ for $ i \in \N$. 

We may concatenate, in a natural way, a string with an infinite sequence. Given an element $\gamma \in X^{+}$ and an infinite sequence $x$ with prefix $\gamma$, we write $x - \gamma$ for the positive infinite sequence $y$ such that $\gamma y = x$.

Taking the product topology on the set $\xno$ gives rise to a space homeomorphic to Cantor space. We introduce some notation for basic open sets in this topology.  Let $\gamma \in \xns$ be a word. We denote by $U_{\gamma}$ the subset of $\xno$ consisting of all elements with prefix $\gamma$.  The set $\{ U_{\gamma} \mid \gamma \in \xns \}$ is a basis for the topology on $\xns$.

We note that the lexicographic ordering extends to a total order on $\xno$.

Define an equivalence relation $\simeqI$ on $ \xno$ as follows. For $x,y \in \xno$ set $x \simeqI y$ if and only if there is a word $w \in \xns$ and $i \in \xn \backslash \{ n-1\}$ such that $x = w i (n-1) (n-1) \ldots$ and $y = w (i+1) 0 0 0 \ldots$. The quotient space $\xno/\simeqI$ is homeomorphic to the interval $[0, 1]$. We often alternate between an element of $\xno$ and its corresponding point in $[0,1]$.

\subsection{Transducers}

In this section we introduce automata and transducers. 

\subsection{Automata and transducers}

An \emph{automaton}, in our context, is a triple $A=(X_A,Q_A,\pi_A)$, where
\begin{enumerate}
	\item $X_A$ is a finite set called the \emph{alphabet} of $A$ (we assume that
	this has cardinality $n$, and identify it with $X_n$, for some $n$);
	\item $Q_A$ is a countable set called the \emph{set of states} of $A$;
	\item $\pi_A$ is a function $X_A\times Q_A\to Q_A$, called the \emph{transition
		function}.
\end{enumerate}
We regard an automaton $A$ as operating as follows. If it is in state $q$ and
reads symbol $a$ (which we suppose to be written on an input tape), it moves
into state $\pi_A(a,q)$ before reading the next symbol. As this suggests, we
can imagine that the inputs to $A$ form a string in $\xn^\mathbb{N}$; after
reading a symbol, the read head moves one place to the right before the next
operation.

We can extend the notation as follows. For $w\in X_n^m$, let $\pi_A(w,q)$ be
the final state of the automaton with initial state $q$ after successively
reading the symbols in $w$. Thus, if $w=x_0x_1\ldots x_{m-1}$, then
\[\pi_A(w,q)=\pi_A(x_{m-1},\pi_A(x_{m-2},\ldots,\pi_A(x_0,q)\ldots)).\]
By convention, we take $\pi_A(\varepsilon,q)=q$.

For a given state $q\in Q_A$, we call the automaton $A$ which starts in
state $q$ an \emph{initial automaton}, denoted by $A_q$, and say that it is
\emph{initialised} at $q$.

An automaton $A$ can be represented by a labelled directed graph, whose
vertex set is $Q_A$; there is a directed edge labelled by $a\in X_a$ from
$q$ to $r$ if $\pi_A(a,q)=r$.

A \emph{circuit} in the automaton $A$ is therefore a word $w \in \xns$ and a state $q \in Q_{A}$ such that $\pi_{T}(w, q) = q$. We say that $w$ is a \emph{circuit based at $q$}. If $|w|= 1$, we say that $w$ is a \emph{loop based at $q$} and $q$ is called a \emph{$w$ loop state}.  A circuit $w$ based at a state $q$ is called \emph{basic} if, writing $w = w_1 \ldots w_{|w|}$,  for all $1 \le i < j < |w|$, $\pi_{T}(w_1\ldots w_i, q) \ne \pi_{T}(w_1\ldots w_{j}, q)$.

A \emph{transducer} is a quadruple $T=(X_T,Q_T,\pi_T,\lambda_T)$, where
\begin{enumerate}
	\item $(X_T,Q_T,\pi_T)$ is an automaton;
	\item $\lambda_T:X_T\times Q_T\to X_T^*$ is the \emph{output function}.
\end{enumerate}
Such a transducer is an automaton which can write as well as read; after
reading symbol $a$ in state $q$, it writes the string $\lambda_T(a,q)$ on an
output tape, and makes a transition into state $\pi_T(a,q)$. An \emph{initial
	transducer} $T_q$ is simply a transducer which starts in state $q$.

In the same manner as for automata, we can extend the notation to allow
the transducer to act on finite strings: let $\pi_T(w,q)$ and $\lambda_T(w,q)$
be, respectively, the final state and the concatenation of all the outputs
obtained when the transducer reads the string $w$ from state $q$.

A transducer $T$ can also be represented as an edge-labelled directed graph.
Again the vertex set is $Q_T$; now, if $\pi_T(a,q)=r$, we put an edge with
label $a|\lambda_T(a,q)$ from $q$ to $r$. In other words, the edge label
describes both the input and the output associated with that edge.

\begin{comment}
For example, Figure~\ref{fig:shift2} describes a transducer over the alphabet
$X_2$.

\begin{figure}[htbp]
	\begin{center}
		\begin{tikzpicture}[shorten >=0.5pt,node distance=3cm,on grid,auto]
		\tikzstyle{every state}=[fill=none,draw=black,text=black]
		\node[state] (q_0)   {$a_1$};
		\node[state] (q_1) [right=of q_0] {$a_2$};
		\path[->]
		(q_0) edge [loop left] node [swap] {$0|0$} ()
		edge [bend left]  node  {$1|0$} (q_1)
		(q_1) edge [loop right]  node [swap]  {$1|1$} ()
		edge [bend left]  node {$0|1$} (q_0);
		\end{tikzpicture}
	\end{center}
	\caption{A transducer over $X_2$ \label{fig:shift2}}
\end{figure}
\end{comment}

A transducer $T$ is said to be \emph{synchronous} if $|\lambda_T(a,q)|=1$
for all $a\in X_T$, $q\in Q_T$; in other words, when it reads a symbol, it
writes a single symbol. More generally, an \emph{asynchronous transducer} may write several symbols, or none at all, at a given
step. 

We can regard an automaton, or a transducer, as acting on an infinite string
in $\xno$, where $X_n$ is the alphabet. This action is given by iterating
the action on a single symbol; so the output string is given by
\[\lambda_T(xw,q) = \lambda_T(x,q)\lambda_T(w,\pi_T(x,q)).\]

Throughout this paper, we make the following
assumption:

\paragraph{Assumption} A transducer $T$ has the property that, when it reads
an infinite input string starting from any state, it writes an infinite  output string.

\medskip

The property above is equivalent to the property that any circuit in the underlying automaton has non-empty concatenated output. 

From the assumption, it follows that the transducer writes an infinite output string on reading any infinite input string from any state.
Thus $T_q$ induces a map $w\mapsto\lambda_T(w,q)$ from $\xno$ to itself; it is
easy to see that this map is continuous. If it is a
homeomorphism, then we call the state $q$ a \emph{homeomorphism state}. We write $\im(q)$ for the image of the map induced by $T_{q}$.

Two states $q_1$ and $q_2$ are said to be \emph{$\omega$-equivalent} if the
transducers $T_{q_1}$ and $T_{q_2}$ induce the same continuous map. (This can
be checked in finite time, see~\cite{GriNekSus}.)  More generally, we say that two
initial transducers $T_q$ and $T'_{q'}$ are \emph{$\omega$-equivalent} if they
induce the same continuous map on $\xno$.

A transducer is said to be \emph{weakly minimal} if no two states are $\omega$-equivalent. 
%For a synchronous transducer $T$, two states $q_1$ and $q_2$ are $\omega$-equivalent if $\lambda_T(a, q_1) = \lambda_T(a,q_2)$ for any finite word $a \in X_n^{*}$. Moreover, if $q_1$ and $q_2$ are $\omega$-equivalent states of a synchronous transducer, then for any finite word $a \in X_{n}^{p}$, $\pi_{T}(a, q_1)$ and $\pi_{T}(a, q_2)$ are also $\omega$-equivalent states. 

For a state $q$ of $T$ and a word $w\in X_n^*$, we let $\Lambda(w,q)$ be the
greatest common prefix of the set $\{\lambda(wx,q):x\in\xno\}$. The state
$q$ is called a \emph{state of incomplete response} if
$\Lambda(\varepsilon,q)\ne\varepsilon$; the length $|\Lambda(\varepsilon,q)|$ of the string $\Lambda(\varepsilon,q)$ is the \emph{extent of incomplete response} of the state $q$. Note that for a state $q \in  Q_{T}$, if the initial transducer $T_{q}$  induces a map from $\xno$ to itself with image size at least $2$, then $|\Lambda(\varepsilon, q)| < \infty$. 

We say that an initial transducer $T_{q}$ is \emph{minimal} if it is weakly minimal, has no states of
incomplete response and every state is accessible from the initial state $q$. A non-initial transducer $T$  is called minimal if for any state $q \in Q_{T}$ the initial transducer $T_{q}$ is minimal.  Therefore a non-initial transducer $T$ is minimal if it is weakly minimal, has no states of incomplete response and is strongly connected as a directed graph.

Two (weakly) minimal non-initial transducers $T$ and $U$  are said to be  \emph{$\omega$-equal} if there is a bijection $f: Q_{T} \to Q_{U}$, such that for any $q \in Q_{T}$, $T_{q}$ is $\omega$-equivalent to $U_{(q)f}$. Two (weakly) minimal initial transducers $T_{p}$ and $U_{q}$ are said to be $\omega$-equal if there is a bijection  $f: Q_{T} \to Q_{U}$, such that $(p)f = q$ and for any $t \in Q_{T}$,  $T_{t}$ is $\omega$-equivalent to $U_{(t)f}$. We shall use the symbol `$=$' to represent $\omega$-equality of initial and non-initial transducers. Two non-initial transducers are said to be $\omega$-equivalent if they have $\omega$-equal minimal representatives.

In the class of synchronous transducers, the  $\omega$-equivalence class of  any transducer has a unique weakly minimal representative.
Grigorchuk \textit{et al.}~\cite{GriNekSus} prove that the $\omega$-equivalence
class of an initialised transducer $T_q$ where every state has finite extent of incomplete response has a unique minimal representative.

Throughout this article, as a matter of convenience, we shall not distinguish between $\omega$-equivalent transducers. Thus, for example, we introduce various groups as if the elements of those groups are minimal transducers and not $\omega$-equivalence classes of transducers. 

Given two transducers $T=(X_n,Q_T,\pi_T,\lambda_T)$ and
$U=(X_n,Q_U,\pi_U,\lambda_U)$ with the same alphabet $X_n$, we define their
product $T*U$. The intuition is that the output for $T$ will become the input
for $U$. Thus we take the alphabet of $T*U$ to be $X_n$, the set of states
to be $Q_{T*U}=Q_T\times Q_U$, and define the transition and rewrite functions
by the rules
\begin{eqnarray*}
	\pi_{T*U}(x,(p,q)) &=& (\pi_T(x,p),\pi_U(\lambda_T(x,p),q)),\\
	\lambda_{T*U}(x,(p,q)) &=& \lambda_U(\lambda_T(x,p),q),
\end{eqnarray*}
for $x\in S_n$, $p\in Q_T$ and $q\in Q_U$. Here we use the earlier 
convention about extending $\lambda$ and $\pi$ to the case when the transducer
reads a finite string.  If $T$ and $U$ are initial with initial states $q$ and $p$ respectively then the state $(q,p)$ is considered the initial state of the product transducer $T*U$.

We say that an initial transducer $T_q$ is \emph{invertible} if there is an initial transducer $U_p$ such that $T_q*U_p$ and $U_p*T_q$ each induce the identity map
on $\xno$. We call $U_p$ an \emph{inverse} of $T_q$.  When this occurs we will denote $U_p$ as $T_q^{-1}$.

In automata theory a synchronous (not necessarily initial) transducer $$T = (X_n, Q_{T}, \pi_{T}, \lambda_T)$$ is \emph{invertible} if for any state $q$ of $T$, the map $\rho_q:=\pi_{T}(\centerdot, q): X_{n} \to X_{n}$ is a bijection. In this case the inverse of $T$ is the transducer $T^{-1}$ with state set $Q_{T^{-1}}:= \{ q^{1} \mid q \in Q_{T}\}$, transition function $\pi_{T^{-1}}: X_{n} \times Q_{T^{-1}} \to Q_{T^{-1}}$ defined by $(x,p^{-1}) \mapsto q^{-1}$ if and only if $\pi_{T}((x)\rho_{p}^{-1}, p) =q$, and output function  $\lambda_{T^{-1}}: X_{n} \times Q_{T^{-1}} \to X_{n}$ defined by  $(x,p) \mapsto (x)\rho_{p}^{-1}$. 

\subsection{From homeomorphisms to transducers}

From the preceding sections we see that every invertible initial transducer induces a homeomorphism on the Cantor space of infinite sequences in the alphabet. In this section exposit a construction in \cite{GriNekSus} for representing any homeomorphism of $\xno$ by a possibly infinite transducers.

We begin by defining some functions.

\begin{Definition}
	Let $\mathrm{lcp}$ be a map defined on the  power set of $\xns \sqcup \xno$ which given a subset  $X$ of $\xns \sqcup \xno$ returns the longest common prefix of the elements in $X$.
\end{Definition}

\begin{Definition}
	Let $h: \xno \to \xno$ be a homeomorphism and $\nu \in \xns$ be a word. We define the \emph{local action} of $h$ at $\nu$ as follows. For $x \in \xno$ we set $(x)h_{\nu}$ to be the element $y \in \xno$ such that $y:= (\nu x) h - \lcp{(U_{\nu})h}$ noting that since $h$ is a homeomorphism $\lcp{(U_{\nu})h}$ is finite.
\end{Definition}

We note that for a homeomorphism $h : \xno \to \xno$, for any word $\nu \in \xns$ the local action $h_{\nu}$ of $h$  at $\nu$ is a continuous injection with clopen image.

\begin{construction}
	Fix a homeomorphism $h: \xno \to \xno$. Let $T$ be the transducer with state set $Q_{T}= \xns$ and transition $\pi_{T}$ and output $\lambda_{T}$ functions  defined as follows.  For $\nu \in \xns$ and $a \in \xn$ 
	we set $\pi_{T}(a, \nu) = \nu_a$ and $\lambda_{T}(a, \nu) = \lcp{(U_{\nu a})h} - \lcp{(U_{\nu})h}$. We note that as the set $(U_{\tau})h$ is clopen for any $\tau \in \xns$, since $h$ is a homeomorphism, then it follows that $T = \gen{\xn, Q_{T}, \pi_{T}, \lambda_{T}}$ is a well-defined transducer.
	 
	It is not hard to show that for any word $\nu \in \xns$, the continuous map of $\xno$ induced by the initial transducer $T_{\nu}$ is in fact equal to the map $h_{\nu}$. 
\end{construction}

\subsection{Synchronizing automata and bisynchronizing transducers}

Given a natural number $k$, we say that an automaton $A$ with alphabet $X_n$
is \emph{synchronizing at level $k$} if there is a map
$\mathfrak{s}_{k}:X_n^k\mapsto Q_A$ such that, for all $q$ and any word
$w\in X_n^k$, we have $\pi_A(w,q)=\mathfrak{s}_{k}(w)$; in other words, if the
automaton reads the word $w$ of length $k$, the final state depends only on
$w$ and not on the initial state. (Again we use the extension of $\pi_A$ to
allow the reading of an input string rather than a single symbol.) We
call $\mathfrak{s}_{k}(w)$ the state of $A$ \emph{forced} by $w$; the map $\mathfrak{s}_{k}$ is called the \emph{synchronizing map at level $k$}. An automaton $A$ is called \emph{strongly synchronizing} if it is synchronizing at level $k$ for some $k$.

We remark here that the notion of synchronization occurs in automata theory
in considerations around the \emph{\v{C}ern\'y conjecture}, in a weaker sense.
A word $w$ is said to be a \emph{reset word} for $A$ if $\pi_A(w,q)$ is
independent of $q$; an automaton is called \emph{synchronizing} if it has
a reset word~\cite{Volkov2008,ACS}. Our definition of ``synchonizing at level $k$''/``strongly synchronizing"
requires every word of length $k$ to be a reset word for the automaton.

If the automaton $A$ is synchronizing at level $k$, we define the
\emph{core} of $A$ to be the set of states forming the image of the map
$\mathfrak{s}_{k}$. The core of $A$ is an automaton in its own right, and is also 
synchronizing at level $k$. We denote this automaton by $\core(A)$. We say that an
automaton or transducer is \emph{core} if it is equal to its core.

It is a result in \cite{BlkYMaisANav} that the product of two strongly synchronizing transducers results in another strongly synchronizing transducer.

Let $T_q$ be an initial transducer which is invertible with inverse $T_q^{-1}$. If $T_q$ is synchronizing at level $k$, and $T_q^{-1}$ is synchronizing at level $l$, 
we say that $T_q$ is \emph{bisynchronizing} at level $(k,l)$. If $T_q$ is
invertible and is synchronizing at level~$k$ but not bisynchronizing, we say
that it is \emph{one-way synchronizing} at level~$k$.

For a non-initial synchronous and invertible transducer  $T$ we also say $T$ is bi-synchronizing (at level $(k,l)$) if both $T$ and its inverse $T^{-1}$ are synchronizing at levels $k$ and $l$ respectively.

%\begin{notation}
	%Let $T$ be a transducer which is synchronizing at level $k$ and Let $l \ge k$ be any natural number. Then for any word $\Gamma \in X_{n}^{l}$, we write $q_{\Gamma}$ for the state $\mathfrak{s}_{l}(\Gamma)$, where $\mathfrak{s}_{l}: X_{n}^{l} \to Q_{T}$ is the synchronizing map at level $l$.
%\end{notation}

\subsection{Groups and monoids of Transducers}

We define several monoids and groups whose elements consists of $\omega$-equivalence classes of transducers which appear first in the papers \cite{BlkYMaisANav,BleakCameronOlukoya2, OlukoyaAutTnr}.  In practice, as each $\omega$-equivalence class contains a unique minimal element as exposited above, we represent the elements of these monoids  by minimal transducers.  

\subsubsection{The monoid \texorpdfstring{$\SOn$}{Lg} and some submonoids}

Let $\SOn$ be the set of all {\bfseries{non-initial}}, minimal, strongly synchronizing and core transducers satisfying the following conditions:

\begin{enumerate}[label = {\bfseries S\arabic*}]
	\item for any state $t \in Q_{T}$, the initial transducer $T_{t}$ induces a injective map $h_{t}: \xno \to \xno$ and,
	\item for any state $t \in Q_{T}$ the map $h_{t}$ has clopen image which we denote by $\im(t)$.
\end{enumerate}
The single state identity transducer is an element of $\SOn$ and we denote it by   $\id$. 

Following \cite{BleakCameronOlukoya2} we define a binary operation on the set $\SOn$ as follows. For two transducers $T,U \in \SOn$, the product transducer $T \ast U$ is strongly synchronizing. Let $TU$ be the transducer obtained from $\core(T\ast U)$ as follows. Fix a state $(t,u)$ of $\core(T\ast U)$. Then $TU$ is the core of the minimal representative of the initial transducer $\core(T\ast U)_{(t,u)}$. 

It is a result of \cite{BleakCameronOlukoya2} that $\SOn$ together with the binary operation  $\SOn \times \SOn \to \SOn$ given by $(T,U) \mapsto TU$ is a monoid.

A consequence of a result in \cite{BlkYMaisANav} is that an element  $T \in \SOn$ is an element of $\SOn$ if and only if  for any state $q \in Q_T$ there is a minimal, initial transducer $U_{p}$ such that the minimal representative of the product $(T\ast U)_{(p,q)}$ is strongly synchronizing and has trivial core.  If the transducer $U_{p}$ is strongly synchronizing as well then $\core(U_{p}) \in \SOn$ and satisfies $T\core(U_p) = \core(U_p)T =\id$. 

The group $\On$ is the largest inverse closed subset of $\SOn$. 

Let $\TSOn$ be the submonoid of $\SOn$ consisting of those elements $T$ which satisfy the following additional constraint:

\begin{enumerate}[label = {\bfseries T\arabic*}]
\item for every state $t \in Q_{T}$ $h_{T}: \xno \to \xno$ preserves the lexicographic ordering.
\end{enumerate}

Write $\TOn{n}$ for the group $\On \cap \TSOn$.

We give below a procedure for constructing the inverse of an element $T \in \SOn$.

\begin{construction}\label{construction:inverse}
	Let $T \in \On$. For $q \in Q_{T}$ define a map $L_{q}: \xnp \to \xns$ by setting $L_{q}(w)$ to be the greatest common prefix of the set $(U_{w})h_{q}^{_1}$.  Set $Q_{T'}$ to be the set of pairs $(w, q)$ where $w \in \xns$, $q \in Q_{T}$, $U_{w} \subset \im{q}$ and $(w)L_{q}   = \ew$. Define maps $\pi_{T'}: \xn \times Q_{T'} \to Q_{T'}$ and $\lambda_{T'}: \xn \times Q_{T'} \to \xns$ as follow:  $\pi_{T'}(a, (w,q)) = wa - \lambda_{T}((wa)L_{q}, q)$ and $\lambda_{T'}(a, (w,q)) = L_{q}(wa)$. Let $T' = \gen{ \xn, Q_{T'}, \pi_{T'}, \lambda_{T'}}$. It is shown in \cite{BlkYMaisANav} that $T'$ is a finite transducer and, setting  $U$ to be the minimal representative of $T'$, $UT = TU = \id$.
\end{construction}

Let $\SLn{n}$ be those elements $T \in \SOn$ which satisfy the following additional constraint:

\begin{enumerate}[label= {\bfseries SL\arabic*}]
	\setcounter{enumi}{2}
	\item for all strings $a \in X_n^{\ast}$ and any state $q \in Q_{T}$ such that $\pi_{T}(a,q) = q$, $|\lambda_{T}(a,q)| = |a|$.
	\label{Lipshitzconstraint}
\end{enumerate}

The set $\SLn{n}$ is a submonoid of $\SOn$ (\cite{BlkYMaisANav}). Set $\Ln{n} := \On \cap \SLn{n}$, the largest inverse closed subset of $\SLn{n}$. Likewise set  $\TSLn = \SLn{n} \cap \TSOn$ and $\TLn{n} = \Ln{n} \cap \TOn{n}$.

For each $1 \le r \le n-1$ the group $\On$ has a subgroup $\Ons{r}$ where $\Ons{n-1} = \Ons{n}$; we define $\Ln{n,r}:= \Ons{r} \cap \Ln{n}$ and $\TOns{r}:= \Ons{r} \cap \TOn{n}$. The following theorem combines results in from \cite{BlkYMaisANav} and \cite{OlukoyaAutTnr}

\begin{Theorem}
	For $1 \le r \le n-1$, $\Ons{r} \cong \out{G_{n,r}}$ and $\TOms{n}{r} \cong \out{T_{n,r}}$.
\end{Theorem}
\begin{comment}
In subsequent sections we establish a map from the monoid $\pn{n}$ to the group $\On$. Theorem~\ref{t:hed2} and Corollary~\ref{cor:FinftyisoPn} show that elements of the quotient $\aut{\xnz, \shift{n}}/\gen{\shift{n}}$ can be represented by elements of $\pn{n}$, therefore the above mentioned map will be crucial in demonstrating that the group $\aut{\xnz, \shift{n}}/ \gen{\shift{n}}$ is isomorphic to $\Ln{n}$. 

In what follows it will be necessary to distinguish between the binary operation in $\spn{n}$ and the operation in $\SLn{n}$. For two elements $P, R \in \spn{n}$, we write $P \spnprod{n} R$ for the element of $\spn{n}$ which is obtained by taking the full transducer product $P * R$, identifying $\omega$-equivalent states and taking the core of the resulting weakly minimal transducer.
\end{comment}

The paper \cite{BleakCameronOlukoya2} gives a faithful representation of the monoid $\SOn$ in the full transformation monoid $\tran(\rwnl{\ast})$.  

Let $\rotclass{\gamma} \in  \rwnl{\ast}$ and $T \in \SOn$. There is a unique state $q$ of $T$ such that $\pi_{T}(\gamma, q) = q$. Let $\nu \in \wnl{\ast}$ be a prime word such that $\lambda_{T}(\gamma, q)$ is a power of $\gamma$. We set $(\rotclass{\gamma})(T)\Pi = \rotclass{\nu}$.  The map $\Pi: \SOn \to \tran(\rwnl{\ast})$ is an injective homomorphism.

\section{An automata theoretic proof that \texorpdfstring{$\out{T_2} \cong \Z/2\Z$}{Lg}}\label{Section:combiproof}
We begin with a series of lemmas. 

The following two lemmas are from  \cite{OlukoyaAutTnr}, however the first is also a consequence of a result Brin in  \cite{MBrin2}.

\begin{lemma}\label{Lemma:orientationppreservingzerooneloopstatesactasidentityonloops}
Let $T \in \TOn{2}$ be an orientation preserving element. Then the unique state $q_{l(a)}$, $a \in \xt$ of $T$ such that $\pi_{T}(a, q_{l(a)}) = q_{l(a)}$, satisfies, $\lambda_{T}(a, q_{l(a)}) = a$.
\end{lemma}

\begin{Remark}
The single state transducer $\R = \gen{\xt, \{\R\}, \pi_{\R}, \lambda_{\R}}$ where $\pi_{\R}(a, \R) = \R$ and $\lambda_{\R}(a, \R) = 1-a$ is an element of $\TOn{2}$ of order $2$. Moreover, for $T \in  \On$, the element $\R T\R$ of $\TOn{2}$, the conjugate of $T$ by $\R$, is the transducer such that for $q \in Q_{T}$ and $a,b \in \xt$, $\pi_{\R T\R}(a, (\R,q,\R)) = (\R,P,\R)$ if and only if $\pi_{T}(1-a, q) = p$ and $\lambda_{\R T\R}(a, (\R,q,\R)) = b$ if and only if $\lambda_{T}(1-a, q) = 1-b$. Thus $\R T\R$ is and orientation preserving element of $\TOn{2}$ if and only if $T$ is.
\end{Remark}

We shall use the remark above, to convert statements about how a state acts when processing zeroes, to how it acts when processing ones. We can also deduce information about the loop state $q_{l(1)}$ by studying properties of the loop state $q_{l(0)}$ with the remark above as well.

\begin{lemma}\label{Lemma:imageofstateclosedsubinterval}
Let $T \in \TOn{2}$ be an orientation preserving element. Let $q$ be any state of $T$ and suppose $\im(q) = U_{\nu_1} \ldots U_{\nu_{m_q}}$  for $\nu_1 \lelex \nu_2 \lelex \ldots \lelex \nu_{m_q}$ elements of $\xtp$. Then, for all $1 \le i < m_{q}$, $\nu_{i}1^{\omega} \simeqI \nu_{i+1}0^{\omega}$  and $\lcp{\{ \nu_i \mid 1 \le i \le m_{q} \}} = \ew.$ 
\end{lemma}

\begin{Remark}
The above lemmas imply that for $T \in \TOn{2}$ an orientation preserving element, for any state $q \in Q_{T}$, its image $\im(q)$ corresponds to a closed dyadic subinterval interval of $I$ whose interior contains the point $\sfrac{1}{2}$. 
\end{Remark}

\begin{lemma}\label{Lemma:characterisingcollapsoftwostates}
Let $T \in \TOn{2}$ be an orientation preserving element. Let $q_1$ and $q_2$ be distinct states of $T$ such that $\pi_{T}(a, q_1) = \pi_{T}(a, q_2) := p_a$ for $a \in \xt$. For $a \in \{0,1\}$, set $\xi_a, \bar{\xi}_a, \eta_a, \bar{\eta}_a \in X_{2}^{\ast}$, such that, $\xi_a$ does not have a suffix equal to $0$, $\eta_a$ does not have a suffix equal to $1$, for $E \in \{ \xi, \eta\}$,  $E_a $ and $\bar{E}_a$ are either both empty or differ only in their last digit, $\xi_a0^{\omega}$ is the smallest element of $\im(p_a)$ and $\eta_a1^{\omega}$ is the largest element of $\im(p_a)$. Then, there are numbers $j_1, j_2, i_1, i_2 \in \N$, such that  one of the following things holds:
\begin{enumerate}[label=(\alph*)]
	\item $\lambda_{T}(0, q_1) = \lambda_{T}(0, q_2) = \ew$, $U_0 \subset \im(p_1)$, $U_1 \not\subset \im(p_0)$, $\eta_0, \bar{\eta}_{0} \ne  \varepsilon$, for $b \in \{1,2\}$, $\lambda_{T}(1, q_b) = \bar{\eta_0}0^{j_b}$,  and $j_1 \ne j_2$;
	\item  $\lambda_{T}(1, q_1) = \lambda_{T}(1, q_2) = \ew$, $U_1 \subset \im(p_0)$, $U_0 \not\subset \im(p_1)$, $\xi_1, \bar{\xi}_{1} \ne \varepsilon$, for $b \in \{1,2\}$, $\lambda_{T}(0, q_b) = \bar{\xi_1}1^{i_b}$, and $i_1 \ne i_2$
	\item $U_1 \subset \im(p_0)$, $U_0 \subset \im(p_1)$, for $b\ \in \{1,2\}$, $\lambda_{T}(0, q_b) = 01^{i_b}$, $\lambda_{T}(1, q_b) = 10^{j_b}$ and either $i_1 \ne i_2$ or $j_2 \ne j_2$. \label{alternativethree}
\end{enumerate}
 %$U_{1} \subset \im(p_0)$, $U_{0} \subset \im(p_1)$, and, for $b \in \{1,2\}$, $\lambda_{T}(0, q_b) = 01^{i_b}$ and $\lambda_{T}(1, q_b) = 10^{j_b}$.    
\end{lemma}
\begin{proof}
First suppose that $\lambda_{T}(0, q_1) = \ew$. Note that under this assumption $p_0$ is not  a homeomorphism state otherwise we have a contradiction of the injectivity of $q_1$. 

%Let $\xi_0, \xi_{1} \in \xts$, such that $\xi_0$ does not end in $0$, $\xi_{1}$ does not end in $1$, and $\xi_0 0^{\omega}$ and $\xi_{1} 1^{\omega}$ correspond, respectively, to the left and right endpoints of the closed subinterval of $I$ corresponding to $\im(p_0)$. 

We observe that $\eta_0 \ne \ew$ otherwise the state $q_{1}$ is either not injective or does not preserve the lexicographic ordering yielding a contradiction. In fact $\eta_0$ has $1$ as a prefix. 

Let $\eta'_{0} \in \xtp$ be such that $\eta'_{0}0 = \eta_0$ and let $\phi = \lcp{\{ \delta \in  X_{2}^{\omega} \mid \eta'_{0} 10^{\omega} \leqlex \delta\}}$. Then $\phi$ has a prefix equal to $1$ and, by minimality of $T$,  is a prefix of $\lambda_{T}(1, q_1)$. Moreover, Lemma~\ref{Lemma:imageofstateclosedsubinterval}, implies that $\lambda_{T}(1,q_1)\xi_1 0^{\omega} = \eta'_010^{\omega}$.  
%if $\eta_0 \in \Xns$ is such that, $\eta_0$ does not have $0$ as a suffix and $\eta_0 0^{\omega}$ is  the left end point of $\im(p_1)$, then we have $\lambda_{T}(1,q_1)\eta_0 0^{\omega} = \bar{\xi}_{1} 1 0^{\omega}$. 

Suppose that $\lambda_{T}(0, q_2) \ne \ew$. As $T$ is minimal, and there is a word $\tau \in \xtp$ such that $U_{0\tau} \subset \im(q_2)$, there is a word $w \in \xts$ such that $\lambda_{T}(0, q_2) = 0w$. Since $\pi_{T}(0,q_2) = p_0$, it follows that the largest (in the lexicographic ordering) element  of the set $\{ \lambda_{T}(0\delta, q_2) \mid \delta \in X_{2}^{\omega} \}$ is equal to $0w \eta'_{0}01^{\omega}$.  Minimality of $T$, and the facts that $\im(q_2)$ corresponds to a closed dyadic subinterval of $[0,1]$ and $1^{\omega}$ is not an element  of $\im(p_0)$, imply that $\lambda_{T}(1, q_2) = \ew$.

We claim that  $\lambda_{T}(1,q_1)$ is a proper prefix of  $\eta_{0}'1$. For suppose this is not the case, then $\lambda_{T}(1, q_1) =   \eta_{0}'10^{j}$ for some $j \in \N$. Since $U_{\eta'_{0}0} \subset \im(p_0)$, then it must be the vase that $U_0 \subset \im(p_1)$. However this contradicts injectivity of  $q_2$ as $\lambda_{T}(0,q_2) = 0w$ and $\lambda_{T}(1, q_2) = \ew$.

Since $\lambda_{T}(1,q_1)$ is a proper prefix of $\eta'_{0}1$, it follows that $\xi_1$,  is a proper suffix of $\eta'_{0}1$. Thus, we have $|0w\eta'_{0}0| - |\xi_1| \ge 2$. Moreover, $\xi_1$ must be a prefix (and so a proper prefix) of $0w \eta'_{0}1$  otherwise we contradict Lemma~\ref{Lemma:imageofstateclosedsubinterval} since $0w\eta'_{0}01^{\omega}$ is the largest element of $\{ \lambda_{T}(0\delta, q_2) \mid \delta \in X_{2}^{\omega} \}$ and $\lambda_{T}(1, q_2) = \ew$. However, the smallest element of $\im(p_1)$ is the point $\xi_1 0^{\omega}$, and as $\xi_1$ is a proper prefix of  $0w \eta'_{0}1$, this contradicts injectivity of the state $q_2$.

 Thus we conclude that $\lambda_{T}(0, q_2) = \ew$ as well. However since $\bar{\eta}_{0}0^{\omega} = \lambda_{T}(1, q_1)\xi_1 0^{\omega} = \lambda_{T}(1, q_2) \xi_1 0^{\omega}$. Then, we either have $\lambda_{T}(1, q_1) = \lambda_{T}(1, q_2)$ which contradicts the fact that $T$ is minimal and $q_1 \ne q_2$, or there are numbers $j_1, j_2 \in \N$, with $j_1 \ne j_2$ such that $\lambda_{T}(1, q_1) = \bar{\eta}0^{j_1}$ and $\lambda_{T}(1, q_2) = \bar{\eta}0^{j_2}$. In the latter case, we see that $U_0 \subset \im(p_1)$.
 
 %In however, this now determines the value of $\lambda_{T}(1, q_2) = \phi = \lcp{\{ \delta \in  \CCmr{2} \mid \bar{\xi} 10^{\omega} \leqlex \delta\}}$. This implies that the states $q_1$ and $q_2$ are $\omega$-equivalent states of $T$, and so $q_1 = q_2$ by minimality of $T$. Yielding a contradiction.

Conjugating by the element $\R$ we deduce that for $T \in \TOn{2}$,   $\lambda_{T}(1, q_1) = \ew$ if and only if $\lambda_{T}(1, q_2) = \ew$ and the following things hold: $\xi_1 \ne \varepsilon$, $U_1 \subset \im(p_0)$, $U_0 \not\subseteq \im(p_1)$ there are number $i_1, i_2 \in \N$, such that  for $b \in \{1,2\}$, $\lambda_{T}(0, q_b) = \bar{\xi}_{1}1^{i_b}$, and $i_1 \ne i_2$. 

Thus we may assume that for $a \in \{0,1\}$ and $b \in \{1,2\}$, $\lambda_{T}(a, q_b) \ne \ew$.

To conclude the lemma, we again make use of the observations that the image of  a state of $T$ corresponds to a closed dyadic interval such that $\sfrac{1}{2}$ is a point in its interior. Thus we observe that for $b \in \{1,2\}$, we have $\tau_{b} =\lambda_{T}(0, q_b)$ is a word beginning with $0$ and $\nu_b:= \lambda_{T}(1, q_b)$ is a word beginning with $1$. If  $\tau_{b}1^{\omega} \not \simeqI \nu_b0^{\omega}$, then the image of $q_b$ does not correspond to a closed dyadic interval. This forces that $\tau_b = 01^{i_b}$ and $\nu_b = 10^{j_b}$ for some $i_b, j_b \in \N$. However, this forces that $1^{\omega} \in \im(p_0)$ and $0^{\omega} \in \im(p_1)$. The statement \ref{alternativethree} of the lemma can now be deduced from these facts.
\end{proof}

The proof of the following lemma from \cite{BleakCameronOlukoya1}  follows straight-forwardly from the definitions.

\begin{lemma}\label{lem:twolettersonestate}
	Let $T$ be a strongly synchronizing automaton and let $k$ be the minimal synchronizing level of $T$. Suppose that there are distinct elements $x,y \in \Xn$ and states $p_1, p_2, p \in Q_{T}$ such that $\pi_{T}(x, p_1) = \pi_{T}(y, p_2)$. Then for any word of $\gamma$ length $k-1$, the maps $\pi_{T}(x\gamma, \cdot): Q_{T} \to Q_{T}$ and $\pi_{T}(y\gamma, \cdot): Q_{T} \to Q_{T}$ have the same image.
\end{lemma}

\begin{lemma}\label{lem:forcingsynch1}
	Let $T \in \TOn{2}$.  Suppose there are states $p_1, p_2, q \in Q_{T}$, and words $\mu_1, \mu_2$, such that $q$ is a homeomorphism state, and, for $a \in \{1,2\}$, $\pi_{T}(\mu_a, p_a) = q$ and $\lambda_{T}(\mu_a, p_a)$ has $a-1$ as a suffix. Then, in fact  $T$ has size one.
\end{lemma}
\begin{proof}
	Let $w \in X_{2}^{\ast}$ and $t \in Q_{T}$ be such that $(w, t)$ is a state of $T'$. For $a \in \{1,2\}$, let  $\delta_a \in X_{2}^{\ast}$ be a word such that $\pi_{T}(\delta_a, t) = p_a$ and  $\lambda_{T}(\delta_a, t)$ has prefix $w$. Then we have $(\lambda_{T}(\delta_a, t)\lambda_{T}(\mu_a,p_a))L_{t} = \delta_a \mu_a$ since $q$ is a homeomorphism state. This implies that $\pi_{T'}(\lambda_{T}(\delta_a, t) \lambda_{T}(\mu_a,p_a) - w, (w,t)) = (\ew, q)$ which is independent of $a$. Thus in $T'$, and so in $T^{-1}$, it is possible to read a word ending in $0$ and a word ending in $1$ to the same location.
	
	The statement of the lemma is now a consequence of Lemma~\ref{lem:twolettersonestate}, since $T^{-1}$ must have minimal synchronizing level $0$, and so $|T| = |T^{-1}| = 1$.
\end{proof}

\begin{lemma}\label{lem:restrictionsonleftend}
	Let $T \in \TOn{2}$ and let $q_1, q_2 \in Q_T$ be distinct states. Suppose there are  words $\mu, \nu_0, \nu_1, \psi, \bar{\psi} \in X_{2}^{\ast}$ and $j_1, j_2 \in \N$ such that
	\begin{enumerate}[label=(\roman*)]
		\item  $\mu \notin \{1\}^{\ast}$, $\pi_{T}(\mu, q_1) = \pi_{T}(\mu, q_2)$,
		\item  $1$ is not a suffix of $\psi$, $0$ is not a suffix of $\bar{\psi}$, $\psi 1^{\omega}$ is the right-most point of $\pi_{T}(\mu, q_1)$, $\psi 1^{\omega}$ and $\bar{\psi}0^{\omega}$ correspond to the same point on the unit circle, and,
		\item for $a \in \{1,2\}$, $\lambda_{T}(\mu1^{\omega}, q_a) = \nu_a0^{j_a}\psi 1^{\omega}$. 
	\end{enumerate}
	     If $\mu'$ is the smallest element in the shortlex ordering on  $X_{2}^{\ast}$ strictly bigger than $\mu$ and  satisfying $\pi_{T}(\mu', q_1) = \pi_{T}(\mu', q_2)$, then the following things hold:
	\begin{enumerate}
		\item $\lambda_{T}(\mu', q_1)$ is a proper prefix of $\nu_1 0^{j_1}\bar{\psi}$ if and only if $\lambda_{T}(\mu', q_2)$ is a proper prefix of $\nu_2 0^{j_2}\bar{\psi}$. In this case $\nu_1 0^{j_1}\bar{\psi} - \lambda_{T}(\mu', q_1) = \nu_2 0^{j_2}\bar{\psi} - \lambda_{T}(\mu', q_2)$. \label{lem:firststatement}
		\item   $\nu_1 0^{j_1}\bar{\psi}$ is prefix of  $\lambda_{T}(\mu', q_1)$ if and only if $\nu_2 0^{j_2}\bar{\psi} $ is a prefix of $\lambda_{T}(\mu', q_2)$. In this case $\lambda_{T}(\mu', q_1) - \nu_1 0^{j_1}\bar{\psi}, \lambda_{T}(\mu', q_2) - \nu_2 0^{j_2}\bar{\psi}  \in \{0\}^{*}$ and $U_{0} \subset \im(\pi_{T}(\mu',q_1)) = \im(\pi_{T}(\mu', q_2))$. \label{lem:secondstatement}
	\end{enumerate}
	 
\end{lemma}
\begin{proof}
	First notice that, since  $q_1$ and $q_2$ are distinct, $\mu \ne \ew$ and, as $\mu$ is not a string of $1$'s, $\mu'$ exists. For $a \in \{1,2\}$, the condition on $\mu'$ implies that $\lambda_{T}(\mu'0^{\omega}, q_a) = \nu_a 0^{j_a}\bar{\psi} 0^{\omega}$. Also note that $\psi$ is the empty word if and only $\bar{\psi}$ is the empty word, and whenever $\psi$ is not  the empty word, $\psi$ and $\bar{\psi}$ differ only in their last letter.

	Set $p = \pi_T(\mu', q_a)$. 
	
	If $\lambda_{T}(\mu', q_1)$ contains $\nu_1 0^{j_1}\bar{\psi}$ as a prefix, then it is clear that $U_0 \subset \im(p)$. Thus if $\lambda_{T}(\mu', q_2)$ is a proper prefix of $\nu_2 0^{j_2}\bar{\psi}$, we have a contradiction of the injectivity of $q_2$. Therefore it follows that there are $i_0, i_1 \in \N$ such that, for $a \in \{1,2\}$, $\lambda_{T}(\mu', q_a) = \nu_a 0^{j_a}\bar{\psi}0^{i_a}$. Swapping $q_1$ with $q_2$ gives statement \ref{lem:secondstatement}.
	
	Thus suppose that $\lambda_{T}(\mu', q_1)$ is a proper prefix of $\nu_1 0^{j_1}\bar{\psi}$. By the statement \ref{lem:secondstatement}, which is proven above, it must be the case that $\lambda_{T}(\mu', q_0)$ is also  a proper prefix of $\nu_0 0^{j_0}\bar{\psi}$.  Let $\rho_1 = \nu_1 0^{j_1}\bar{\psi} - \lambda_{T}(\mu', q_1)$.  Then  it is clear that $\rho_1 0^{\omega}$ is the left-most point of $\im(p)$.   However, similarly defining $\rho_0 = \nu_0 0^{j_0}\bar{\psi} - \lambda_{T}(\mu', q_1)$, it is clear that $\rho_0 0^{\omega}$ is also the left-most point of $\im(p)$. Thus we must have that $\rho_0 = \rho_1$ since they both end in $1$. This proves statement \ref{lem:firststatement}.

\end{proof}

\begin{Theorem}
	Let $T \in \TOn{2}$, then $|T| = 1$. 
\end{Theorem}
\begin{proof}
	It suffices to prove that if $T \in \TOn{2}$ is an orientation preserving element, then $|T| = 1$.
	
	Suppose for a contradiction that $|T| > 1$. Let $k, j \in \N$ be the minimal synchronizing levels of $T$ and $T^{-1}$ respectively, and let $\T{W}$ be the set of those elements $w \in X_{2}^{\ast}$, of length at most  $k$ which satisfy the following requirements:
	
	\begin{itemize}
		\item $\pi_{T}(w, q_{l(0)}) \ne  \pi_{T}(w, q_{l(1)})$ and,
		\item for any $x \in X_{2}$, $\pi_{T}(wx, q_{l(0)}) =  \pi_{T}(wx, q_{l(1)})$.
	\end{itemize}

By conjugating $T$ by $\R$ if necessary, since $q_{l(0)}$ and $q_{l(1)}$ are distinct states, we may further assume that there is a word $w \in \T{W}$, for which $\lambda_{T}(1, \pi_{T}(w,q_{l(0)})) \ne \lambda_{T}(1, \pi_{T}(w,q_{l(1)}))$. Note that the requirements on $\T{W}$ guarantee that all its elements are pairwise incomparable in the prefix ordering. Fix $w$ in $\T{W}$ to be the largest element in the lexicographic  ordering of $\T{W}$. We now argue that $|T| = 1$.

For $a \in \{0,1\}$, let  $\mu_a = \lambda_{T}(w, q_{l(a)})$, $q_a = \pi_{T}(w, q_{l(a)})$ and $p_a = \pi_{T}(a, q_a)$. By Lemma~\ref{Lemma:characterisingcollapsoftwostates}, for $a \in \{0,1\}$, there are  $k_a \in \N$ and a word $\Xi \in X_{2}^{\ast}$   such that $\lambda_{T}(1, q_{a}) = \Xi10^{k_a}$. By assumption  $k_0 \ne k_1$.  Since $U_{1} \subset \im(q_{l(1)})$, it follows that $U_{\mu_1 1} \subset  \im(q_{l(1)})$.

Observe that  $\lambda_{T}(w11^{\omega},q_{l(a)}) = \mu_{a}\Xi10^{k_a}\lambda_{T}(1^{\omega}, p_1)$. Let $\psi \in X_{2}^{\ast}$, with $1$ not a suffix of $\psi$ be such that $\psi 1^{\omega}$ is the right-most point of $\im(p_1)$. Define $\bar{\psi}$ such that if $\psi$ is not empty, $\psi$ and $\bar{\psi}$ differ only in their last letter, otherwise $\bar{\psi} = \ew$. 

First suppose that $w \in \{1\}^{+}$. In this case, we have have that $p_1 = q_{l(1)}$, $q_1 = q_{l(1)}$, and, by Lemma~\ref{Lemma:characterisingcollapsoftwostates}, $U_0 \subset \im(p_1)$ and so $p_1$ is a homeomorphism state. This means that we are in alternative ~\ref{alternativethree} of Lemma~\ref{Lemma:characterisingcollapsoftwostates}, and so $\Xi = \ew$. Thus, as $\lambda_{T}(1,  q_{l(1)}) = 1$, it follows that $k_1 = 0$ and so, by assumption, $k_0 >0$. As $\lambda_{T}(1, q_0) = 10^{k_0}$, it follows that $q_0, q_{l(1)},1$ satisfy the hypothesis of Lemma~\ref{lem:forcingsynch1}, and so,  $|T| = 1$.

Suppose, now that $w \notin \{1\}^{+}$, we once again deduce the conclusion of Lemma~\ref{lem:forcingsynch1}.
	
Let $v$ be the smallest word in the shortlex ordering on $X_{2}^{\ast}$ bigger than  $w$ and satisfying $\pi_{T}(v, q_{l(0)}) = \pi_{T}(v, q_{l(1)})$. Then, by Lemma~\ref{lem:restrictionsonleftend}, it is the case that either $ \mu_0\Xi10^{k_0}\bar{\psi} - \lambda_{T}(v, q_{l(0)}) = \mu_1\Xi10^{k_1}\bar{\psi} - \lambda_{T}(v, q_{l(1)})$ or there are  $j_0, j_1 \in \N$ such that, for $a \in \{0,1\}$,  $\lambda_{T}(v, q_{l(a)})= \mu_a\Xi10^{k_a}\bar{\psi}0^{j_a}$. It therefore follows that one of the following things holds: $ \mu_0\Xi10^{k_0}\bar{\psi} - \lambda_{T}(v, q_{l(0)}) = \mu_1\Xi10^{k_1}\bar{\psi} - \lambda_{T}(v, q_{l(1)})$, $\lambda_{T}(v, q_{l(1)}) - \mu_0\Xi10^{k_0}\bar{\psi}= \lambda_{T}(v, q_{l(1)}) - \mu_0\Xi10^{k_1}\bar{\psi}$, or $\lambda_{T}(v, q_{l(0)}) - \mu_0\Xi10^{k_0}\bar{\psi}$ and  
$ \lambda_{T}(v, q_{l(1)}) - \mu_0\Xi10^{k_1}\bar{\psi}$ differ by a power of $0$. 

Let $p' = \pi_{T}(v, q_{l(a)})$, $a \in \{1,2\}$ and let $\phi$ be such that $\phi$ has no suffix equal to $1$ an $\phi1^{\omega}$ is the rightmost point of $p'$. Then defining $\bar{\phi}$ analogously to $\bar{\psi}$, it follows that, if $v \notin \{1\}^{+}$, then, $v, q_{l(0)}, q_{l(0)}, \lambda_{T}(v, q_{l(0)}), \lambda_{T}(v, q_{l(1)}), \phi, \bar{\phi}$, satisfy the hypothesis of Lemma~\ref{lem:restrictionsonleftend}.

Therefore, by induction and since  $U_{\mu_11} \subset \im(q_{l(1)})$ there is a  largest element $v$ in the shortlex ordering of $X_{2}^{\ast}$, such that  all of the following hold:
\begin{itemize}
	\item $\pi_{T}(v, q_{l(1)}) = \pi_{T}(v, q_{l(0)})$,
	
	\item $U_{0} \subset \im(\pi_{T}(v, q_{l(a)}))$, $a \in \{0,1\}$,
	
	\item for any proper prefix $\nu$ of $v$, $\pi_{T}(\nu, q_{l(1)}) \ne \pi_{T}(\nu, q_{l(0)})$, 
	\item $\lambda_{T}(v1^{\omega}, q_{l(1)}) \leqlex \mu_1\Xi 1^\omega$, and,
	\item there is a $\phi \in X_{2}^{\ast}$ without $1$ as a suffix, such that  
	if $v'$ is the largest element of $X_{2}^{\ast}$ which is strictly smaller than $v$ in the shortlex ordering, for which  $\pi_{T}(v',q_{l(0)}) = \pi_{T}(v', q_{l(1)})$, no proper prefix $\nu'$ of $v'$ satisfies $\pi_{T}(\nu',q_{l(0)}) = \pi_{T}(\nu', q_{l(1)})$, and $\phi 1^{\omega}$ is the largest element of $\im(\pi_{T}(v', q_a))$, then, defining $\bar{\phi}$ as before, the following inequality is well-defined and holds $\lambda_{T}(v, q_{l(1)}) -\lambda_{T}(v',q_{l(1)})\bar{\phi}  \ne \lambda_{T}(v, q_{l(0)}) - \lambda_{T}(v',q_{l(0)})\bar{\phi}$.
\end{itemize}

Note that since, for $a \in \{0,1\}$, as $U_{0} \subset \im(p_1)$, $\lambda_{T}(w01^{\omega}, q_{l(a)}) = \mu_a\Xi01^{\omega}$, it is possible that $v = w1$.

Let $a \in \{0,1\}$, $k'_{a} \in \N$, $\mu'_{a}$ and $p'$, be such that $\pi_{T}(v, q_{l(a)}) = p'$ and $\lambda_{T}(v, q_{l(a)}) = \mu'_a10^{k'_a}$. By assumption  we have that $k'_1 \ne k'_0$. Let $\phi, \bar{\phi} \in X_{2}^{\ast}$ be such that either $\phi = \bar{\phi} = \ew$ or $\phi$ and $\bar{\phi}$ differ only in their last digit, and $\phi1^{\omega}$ is the right-most point of $\im(p')$. Let $d = \mid k'_1 - k'_0 \mid $, noting, that by properties of $v$, $d >0$. For $a \in \{0,1\}$ set $d_a = k_a$ if $k'_a \le k'_{a+1 \pmod{1}}$, otherwise set  $d_a = k'_{a+1 \pmod{1}}$. Note that $d_0 = d_1 = \min\{k'_1, k'_0\}$.

If $v \in \{1\}^{+}$, then, we may once more apply Lemma~\ref{lem:forcingsynch1} to conclude that $|T| =1$, since, in this case $k_1 = 0$, forcing $k_0 >0$, $p' = q_{l(1)}$, and so $p'$ is a homeomorphism  state, since $U_0 \subset \im(p')$, thus $q_{l(0)}, q_{l(1)}, v, p'$ satisfy the required hypothesis. 

Thus we may assume that $v \not \in \{ 1\}^{+}$. We show, by induction, that in this case, $\mu'_a10^{k'_a - d_a} \subset \im(q_{l(a)})$, and, for $x \in X_{2}^{\omega}$, arbitrary, there is a $y \in X_{2}^{\omega}$, independent of $a$, such that $\lambda_{T}(y, q_{l(a)}) = \mu'_a10^{k'_a -d_a}x$. From this we deduce that $|T| = 1$.

Our induction is a repeated application of the following claim.  
\begin{claim}\label{claim:inductioncore}
	Let $t \in X_{2}^{+}$ be such that $t \notin \{1\}^{+}$, $v \leqlex t$, $ \pi_{T}(t, q_{l(1)}) = \pi_{T}(t, q_{l(0)})=: p''$, for $a \in \{0,1\}$, $\mu'_{a}10^{k'_a - d_a}$ is a prefix of $\lambda_{T}(t, q_{l(a)})$, and, $\lambda_{T}(t, q_{l(0)}) - \mu'_{0}10^{k'_0 - d_0} = \lambda_{T}(t, q_{l(1)}) - \mu'_{1}10^{k'_1 - d_1}$. Let $\phi', \bar{\phi}'$ be defined as usual and such that $\phi'1^{\omega}$ is the right-most point of $\im(p'')$. Suppose that $\lambda_{T}(t, q_{l(1)})\phi'1^{\omega} \lelex \mu'_1 10^{k'_1 - d_1}1^{\omega}$. 
	
	Let $t'$ be the smallest element, in the shortlex ordering on $X_{2}^{\ast}$, strictly bigger than $t$ and  satisfying $\pi_{T}(t', q_{l(1)}) = \pi_{T}(t', q_{l(0)})$. Then for any $x \in X_{2}^{\omega}$, such that, for $a \in \{0,1\}$, $\lambda_{T}(t1^{\omega}, q_{l(a)} ) \lelex \mu'_{a}10^{k'_a-d_a} x \leqlex \lambda_{T}(t'1^{\omega}, q_{l(a)})$, there is a $y \in X_{2}^{\omega}$, independent of $a$, such that for $a \in \{0,1\}$, $\lambda_{T}(y, q_{l(a)}) = \mu'_a 10^{k'_a - d_a}x$. Moreover, for $a \in  \{0,1\}$, $\mu'_{a}10^{k'_a-d_a}$ is a prefix of  $\lambda_{T}(t'1^{\omega}, q_{l(a)})$, and $\lambda_{T}(t', q_{l(0)}) - \mu'10^{k'_0 - d_0} = \lambda_{T}(t', q_{l(1)}) - \mu'_{0}10^{k'_1 - d_1}$.
\end{claim}
\begin{proof}
	By Lemma~\ref{lem:restrictionsonleftend}, $\lambda(t', q_{l(1)})$ is a proper prefix of $\lambda_{T}(t, q_{l(1)})\bar{\phi}'$ if and only if $\lambda_{T}(t', q_{l(0)})$ is a proper prefix of $\lambda_{T}(t, q_{l(1)})\bar{\phi}'$. In this case,  $\lambda_{T}(t, q_{l(1)})\bar{\phi}' - \lambda(t', q_{l(1)}) = \lambda_{T}(t, q_{l(0)})\bar{\phi}' - \lambda(t', q_{l(0)})$. Now, as $\max\{k'_1 - d_1, k'_0 - d_0\} > 0$, and $k'_1 \ne k'_0$, and since, for $a \in \{0,1\}$, $\mu'_a 10^{k'_a - d_a}$ is a prefix of $\lambda_{T}(t, q_{l(a)})$, the equalities, $\lambda_{T}(t, q_{l(0)}) - \mu'_{0}10^{k'_0 - d_0} = \lambda_{T}(t, q_{l(1)}) - \mu'_{1}10^{k'_1 - d_1}$ and $\lambda_{T}(t, q_{l(1)})\bar{\phi}' - \lambda(t', q_{l(1)}) = \lambda_{T}(t, q_{l(0)})\bar{\phi}' - \lambda(t', q_{l(0)})$ imply that, for $a \in \{0,1\}$, $\mu'_a 10^{k'_a - d_a}$ is a prefix of $\lambda_{T}(t', q_{l(a)})$. Moreover, by the same equalities, we also have  $\lambda_{T}(t', q_{l(0)}) - \mu'_{0}10^{k'_0 - d_0} = \lambda_{T}(t', q_{l(1)}) - \mu'_{1}10^{k'_1 - d_1}$. 
	
   If  $\lambda_{T}(t, q_{l(1)})\bar{\phi}'$ is a prefix of $\lambda_{T}(t', q_{l(1)})$, then by Lemma~\ref{lem:restrictionsonleftend}, there are $m_0, m_1 \in \N$ such that, for $a \in \{0,1\}$, $\lambda_{T}(t', q_{l(a)}) = \lambda_{T}(t, q_{l(a)})\bar{\phi}'0^{m_a}$. However, by assumption on $v$ and since $v \leqlex t \lelex t'$, we must have that $m_1 = m_0$. Thus, once more $\lambda_{T}(t', q_{l(0)}) - \mu'_{0}10^{k'_0 - d_0} = \lambda_{T}(t', q_{l(1)}) - \mu'_{1}10^{k'_1 - d_1}$, since $\lambda_{T}(t, q_{l(0)}) - \mu'_{0}10^{k'_0 - d_0} = \lambda_{T}(t, q_{l(1)}) - \mu'_{1}10^{k'_1 - d_1}$. 
   
   Now let $x \in X_{2}^{\omega}$ and, for $a \in \{0,1\}$, suppose that $\lambda_{T}(t, q_{l(a)})\phi'1^{\omega} \lelex \mu'_{a} 10^{k'_a-d_a}x \leqlex \lambda_{T}(t'1^{\omega}, q_{l(a)})$. Observe that $L_{q_{l(a)}}(\mu'_a10^{k'_a-d_a}x) = t'L_{p''}(\mu'_a10^{k'_a-d_a}x - \lambda_{T}(t', q_{l(a)})$, since the points  $\lambda_{T}(t, q_{l(a)})\phi'1^{\omega} $ and $\lambda_{T}(t'0^{\omega},q_{l(a)} )$ are equal on the interval. Further notice that as $\lambda_{T}(t', q_{l(0)}) - \mu'_{0}10^{k'_0 - d_0} = \lambda_{T}(t', q_{l(1)}) - \mu'_{1}10^{k'_1 - d_1}$, $\mu'_010^{k'_0-d_0}x - \lambda_{T}(t', q_{l(0)} = \mu'_110^{k'_1-d_1}x - \lambda_{T}(t', q_{l(1)}$. Therefore, $t'L_{p''}(\mu'_110^{k'_1-d_1}x - \lambda_{T}(t', q_{l(1)}) = t'L_{p''}(\mu'_010^{k'_0-d_0}x - \lambda_{T}(t', q_{l(0)})$. Thus, setting $y = t'L_{p''}(\mu'_110^{k'_1-d_1}x - \lambda_{T}(t', q_{l(1)})$, we have, for $a \in \{0,1\}$, $\lambda_{T}(y, q_{l(a)}) = \mu'_{a}10^{k'_a - d_a} x$.
\end{proof}

Since $U_{\mu'_1} \subseteq \im(q_{l(1)})$, there is a $t_{\mathrm{max}}$ which is the least element of $X_{2}^{\ast}$ in the shortlex ordering, such that $v \lelex t_{\mathrm{max}}$, $\pi_T(t_{\mathrm{max}}, q_{l(1)} = \pi_{T}(t_{\mathrm{max}}, q_{l(0)})$, and  $\mu'_1 10^{k_1 - d_1}1^{\omega} \in \lambda_{T}(t_{\mathrm{max}}, q_{{l(1)}})\im(\pi_{T}(t_{\mathrm{max}}, q_{l(1)}))$. Thus, repeated applications of Claim~\ref{claim:inductioncore}, shows that for any $x \in X_{2}^{\omega}$, such that, for $a \in \{0,1\}$, $\mu'_{a}10^{k'_a}\phi 1^{\omega} \le x \le \mu'_{a}10^{k'_a - d_a} 1^{\omega}$, $(x)h_{q_{l(0)}}^{-1} = (x)h_{q_{l(1)}}^{-1}$. Moreover observe that for $x \in  X_{2}^{\omega}$ such that,  $\mu'_{a}10^{k'_a}0^{\omega} \leqlex \mu'_{a}10^{k'_a}x \leqlex \mu_{a}10^{k'_a}\phi1^{\omega}$, $a \in \{0,1\}$, we have $L_{q_{l(a)}}(\mu'_{a}10^{k'_a}x) = v L_{p'}(x)$. Thus it is the case that  for  $\mu'_{a}10^{k'_a}0^{\omega} \le x \le \mu'_{a}10^{k'_a - d_a} 1^{\omega}$, $(x)h_{q_{l(0)}}^{-1} = (x)h_{q_{l(1)}}^{-1}$.

The above paragraph now implies that, $L_{q_{l(1)}}(\mu'_110^{k'_1-d_1}) = L_{q_{l(0)}}(\mu'_010^{k'_0-d_0})$. 

For $a \in \{0,1\}$, let $\nu_a = \mu'_a10^{k'_a-d_a} - \lambda_{T}(L_{q_{l(a)}}(\mu'_a10^{k'_a-d_a}), q_{l(a)})$, and let $R_a = \pi_{T}(L_{q_{l(a)}}(\mu'_a10^{k'_a-d_a}), q_{l(a)})$. Then $(\nu_a, R_a)$ is a state of $T'$, for $a \in \{1,2\}$. Moreover,  since for any $x \in X_{2}^{\omega}$, $L_{q_{l(a)}}(\mu'_a10^{k'_a-d_a}x) = L_{q_{l(1)}}(\mu'_110^{k'_1-d_1})L_{R_a}(\nu_a x)$, it follows that, as $ L_{q_{l(0)}}(\mu'_010^{k'_0-d_0}x) = L_{q_{l(1)}}(\mu'_110^{k'_1-d_1}x)$ for all $x \in X_{2}^{\omega}$, and $L_{q_{l(1)}}(\mu'_110^{k'_1-d_1}) = L_{q_{l(0)}}(\mu'_010^{k'_0-d_0})$, then, $L_{R_0}(\nu_0 x) = L_{R_1}(\nu_1 x)$.  Thus the states $(\nu_1, R_1)$ and $(\nu_0, R_0)$ are $\omega$-equivalent states of $T'$ and so represent the same state of $T^{-1}$.

Let $(\xi, S)$ be any state of $T'$, and, for $a \in \{0,1\}$ let $\delta_a \in X_{2}^{\omega}$, be such that $\lambda_{T}(\delta_a, S)$ has $\xi$ as a prefix,  and $\pi_{T}(\delta_a , S) = q_{l(a)}$. Then notice that $L_{S}(\lambda_{T}(\delta_a,S)\mu'_a10^{k'_a-d_a}) = \delta_aL_{q_{l(a)}}(\mu'_a10^{k'_a-d_a})$. Thus, $\pi_{T'}(\lambda_{T}(\delta_a,S)\mu'_a10^{k'_a-d_a} -\xi, (\xi, S) ) = (\nu_a, R_a)$. Since $\{k_a-d_a \mid a \in \{0,1\}\} = \{0, d\}$, and $d >0$, it follows that it is possible to read  a word ending in $0$ and a word ending in $1$ to the same location  in $T^{-1}$. As $T^{-1}$ is strongly synchronizing this means that, for any word $\gamma \in X_{2}^{j-1}$, the map $\pi_{T^{-1}}(\gamma, \centerdot): Q_{T^{-1}} \to Q_{T^{-1}}$ has only one image. This yields the desired contradiction, since $j$ is assumed to be the minimal synchronizing level of $T^{-1}$. Thus we deduce that $|T| = 1$.

Therefore we see that if $T \in \TOn{2}$, and  $|T|>1$, then we may replace $T$ with a conjugate $\R T \R$ such that $|\R T \R| = 1$. However, $|\R T \R| = 1$ if and only if $|T| = 1$, which yields the desired contradiction.

\end{proof}

\section{The intersection  \texorpdfstring{$\TOn{n} \cap \Ln{n}$}{Lg}}\label{Section:intersection}

The results of the previous section demonstrates that the intersection of $\TOn{2}$ with $\Ln{2}$ is finite and in particular, it is equal to the cyclic group of order $2$. This leads us to consider the intersection $\TLn{n}$ of the subgroups $\TOn{n}$ and $\Ln{n}$ of $\On$. We demonstrate in this section that the group $\TLn{n}$ is $C_{2}$ the cyclic group of order $2$ when $n$ is prime; is isomorphic to the direct product of $C_2$ with a finite cyclic group when $n$ is a power of a prime; and, when $n$ is not a power of a prime, it is isomorphic to the group $ C_2 \times \Z^{\pd{n}-1} \times C_{l}$ where $l$ is a number related to the prime decomposition of $n$ and $\pd{n}$ is the number of distinct primes dividing $n$. 

We require some further results and definitions from the article \cite{BelkBleakCameronOlukoya3}. 

\begin{Definition}
	Let $T \in \SOn$. For a state $q \in Q_{T}$ let $M_{q} \in \N$ be defined as follows. Since $\im(q)$ is clopen, there is a minimal subset $C \subseteq \xns$ such that  $\bigcup_{c \in C} U_{c} = \im(q)$. Set $M_{q} = |C|$ and $m_{q} \in \Z_{n-1}$ such that $M_{q} \equiv m_{q} \pmod{n-1}$.
\end{Definition}

Set $\Un{n-1}$ to be the group of units of $\Z_{n-1}$. The following result is proved in \cite{OlukoyaAutTnr}. 

\begin{Theorem}
	Let $T \in  \SOn$ then for any pair of states $q, q' \in Q_{T}$, $m_{q} = m_{q'}$. Moreover, the map $\rsig: \On  \to \Un{n}$ given by $T \mapsto m_{q}$ for a state $q \in Q$, is a homomorphism.
\end{Theorem}

The following result, which combines a result from \cite{OlukoyaAutTnr} with a result in \cite{BelkBleakCameronOlukoya3} is crucial for the results in this section.

\begin{proposition} \label{thm:conecoveringforastate}
	Let $T \in \Ln{n}$. Then there is a maximal number $s(T)$ a divisor of a power of $n$ which is not divisible by $n$ and such that the following holds. For  any state $q \in Q_{T}$ there are  numbers $l,m \in \N_{0}$, and elements $\mu_1, \mu_2, \ldots, \mu_{s(T)n^{m}} \in \xn^{l}$ such that $\im(q) = \bigcup_{1 \le a \le s} U_{\mu_a}$. Moreover, for $T, U \in \Ln{n}$,  if $s$ is the maximal number not divisible by $n$  $s(T)s(U) = sn^{j}$ for some $j \in \N$, then $s(TU) = s$.
\end{proposition}

Define a group $\Gn{n}$ as follows. Let $G_{n}$ be the submonoid of $\N$ generated by the prime divisors of $n$. Define an equivalence relation $\sim$ on $G_{n}$ such that $ u \sim t$ if and only if $u$ is a power of $n$ times $t$. Then $\Gn{n} := G_{n} / \sim$ is a finitely generated abelian group. It is shown in \cite{BelkBleakCameronOlukoya3} that $\Gn{n} \cong \Z^{\pd{n} -1} \times \Z/l\Z$ where, if $ p_1^{l_1} p_{2}^{l_2} \ldots p_{\pd{n}}^{l_{\pd{n}}}$ is the prime decomposition of $n$, then $l = \gcd(l_1, \ldots, l_r)$.

The following result is from \cite{BelkBleakCameronOlukoya3} and is essentially a corollary of Proposition~\ref{thm:conecoveringforastate}.

\begin{Theorem}\label{Thm:epitoGn}
	The map $s: \Ln{n} \to \Gn{n}$ by $T \mapsto s(T)$ is an epimorphism.
\end{Theorem}

We also require the result below which characterises when an element of $\SOn$ belongs to $\On$ based on the map $\Pi$.

\begin{Theorem}[Bleak, Cameron, O]
Let $T \in \SOn$ and suppose that $(T)\Pi$ is a bijection. Then $T \in \On$.
\end{Theorem}

As an immediate corollary we have:

\begin{corollary}\label{cor:TMnbijectionTOn}
	Let $T \in \TSOn$ and suppose that $(T)\Pi$ is a bijection. Then $T \in \TOn{n}$.
\end{corollary}

We begin with the following lemma.

\begin{lemma}\label{lem:powerofnmeanssizeone}
	Let $T \in \TLn{n}$. Then $s(T) = 1$ if and only if $|T| = 1$.
\end{lemma}
\begin{proof}
	Clearly if $|T|= 1$ then $T$ has a single state inducing a homeomorphism of the interval and so $s(T) = 1$.
	
	Let $T \in \TOn{n}$ be an orientation preserving element.
	
	First assume that there is a state $q \in Q_{T}$ that is a homeomorphism state. Since $U_{0} \subseteq \im(q)$, it follows that $\lambda_{T}(0, q)$ is a power of $0$. Set $p_i = \pi_{T}(i, q)$ for $0 \le i \le n-1$. We note that $U_{0} \subseteq U_{p_0}$ since $T$ has no states of incomplete response and $U_0 \subseteq \im(q)$.
	
	Now let $l, m \in \N_0$ and $\mu_1 \lelex \mu_2 \lelex \ldots \lelex \mu_{n^{m}} \in \xn^{l}$ be such that $\im(p_0) = \bigcup_{1 \le a \le n^{m}}U_{\mu_a}$. Since $U_{0} \subset \im p_0$, it must be the case that $n^{l-1} \le s$ and in fact,  $\bigcup_{1 a \le n^{l-1}} U_{\mu_a} = U_{0}$. However, since $T$ has no states of incomplete response, it must be the case that $U_{1} \cap \im(p_0) \ne \emptyset$. This means that $m = {l}$ and so $\mu_1 \lelex \mu_2 \lelex \ldots \lelex \mu_{n^{l}} = \xn^{l}$. Thus $p_0$ is a homeomorphism state.
	
	However, since the image of every state of $T$ induces an orientation preserving continuous  injection of the interval, the state $q$ induces a homeomorphism of the interval. In particular $U_{0} \subseteq \im(p_1)$.
	
	By repeating the arguments above, we see that this implies that $p_1$ is a homeomorphism state. An easy induction argument now implies that $p_i$ is a homeomorphism state for every $i$.
	
	Thus if $T$ contains a homeomorphism state, then all states of $T$ are homeomorphism states.
	
	We now argue that the unique state $q_0 \in Q_{T}$ which satisfies $\pi_{T}(0, q_0) = q_0$ is a homeomorphism state. To see this, observe that by Lemma~\ref{Lemma:orientationppreservingzerooneloopstatesactasidentityonloops} we have $\lambda_{T}(0, q_0) = 0$, in particular, since $T$ is assumed minimal, $U_0 \subset \im(q_0)$. By preceding arguments we then conclude, using Theorem~\ref{thm:conecoveringforastate} again, that $q_0$ is a homeomorphism state.
	
	Thus all states of $T$ are homeomorphism states, in particular $T$ is an invertible, synchronous transducer. This means that all states of $T$ induce bijections of $\xn$. Since $T$ is an orientation preserving element of $\TOn{n}$, all states of $T$ must induce the identity map on $\xn$. Therefore, $T$ is the single state identity transducer.
	 
	The result now follows. 
\end{proof}

When $n$ is not prime there are elements of $\TLn{n}$ which have minimal representatives of size bigger than $1$. We construct such elements below.

\begin{example}\label{Example:Tnuntodivisorsofn}
	Let $d,e \in \N$ be such that $de = n$. Construct a transducer $T(d,e)$ as follows. The state sate $Q_{T(d,e)} = \{ q_0, q_1, \ldots, q_{d-1}\}$. The transition and output functions are gives as follows. Let $0 \le i \le s-1$ and $0 \le j \le  d-1$, write $i = ad + b$ for some $0 \le b \le d-1$ and $0 \le a \le e-1$, we have $\pi_{T(d,e)}(ad + b, q_{j}) = q_{b}$ and $\lambda_{T(d,e)}(ad+b, q_{j}) = je + a$.
	
	We note that $T(d,e) $ is strongly synchronizing at level $1$. 
	
	Fix a state $q_{j}$ for $0 \le j \le d-1$. First note that the image of $q_{j}$ is precisely the union $\bigcup_{1 \le a \le s-1} U_{je +a}$. This follows by induction using the following observation: for any $0 \le a \le e-1$, $x = le + b$, $0 \le l \le d-1$, $0 \le b \le e-1$, $\pi_{T}(ad+l, q_j) = q_l$ and $\lambda_{T}(ad+l, q_j)= je + a$.
	
	We note also that all states of $T(d,e)$ are injective. For any state $q_j$ of $T(d,e)$,  $\lambda_{T(d,e)}(x, q_j) = \lambda_{T(d,e)}(y,q_j)$ for distinct $x, y \in \xn$ if and only if there are $0 \le a \le e-1$ and distinct $0 \le b, c \le d-1$ such that  $x = ad + b$ and $y = ad + c$. In this case we therefore have $\pi_{T(d,e)}(x, q_j) = q_b$ and $\pi_{T(d,e)}(y, q_j) = q_{c}$. We then note that $\im(q_b) \cap \im(q_c) = \emptyset$ by the observation in the preceding paragraph.
	
	Let $x,y \in \xn$. Write $x = je +a$ and $y = je + b$. Then the state $q_j$ is the unique state for which there are $u,v \in \xn$ with $\lambda_{T(d,e)}(uv, q_j) = xy$. Thus, as $T(d,e)$ is synchronous, we see that the induced map $(T(d,e))\Pi$ on $\rwnl{\ast}$ is in fact a bijection.
	
	Lastly notice that since for $j \le l$, $\im(q_j) \lelex \im(q_{l})$, then each state of $T(d,e)$ preserves the lexicographic ordering and so induces a continuous injection of the interval. Thus, by Corollary~\ref{cor:TMnbijectionTOn}, we conclude that $T(d,e)$ is an orientation preserving element of $\TOn{n}$. Note, moreover, that, as observed above, $s(T(d,e)) = e$.
\end{example}

We now prove the main result of this section.

\begin{Theorem}
	The group $\TLn{n} = \TOn{n} \cap \Ln{n}$ is precisely the subgroup $$\gen{\T{R}, T(\sfrac{n}{p},p) : p \mathrm{ \ is \ a \ prime \ dividing \ }n}$$ generated by $\T{R}$ and the elements $T(\sfrac{n}{p}, p)$ for $p$ a prime divisor of $n$. Furthermore, writing $n = p_1^{l_1}p_2^{l_2} \ldots p_{r}^{l_r}$, where the $p_i$ are the distinct prime divisors of $n$, and setting $l = \gcd(l_1, \ldots, l_r)$, we have: 
	\[
		 \TLn{n} \cong C_2 \times \Z^{\pd{n}-1} \times C_{l}.
	\]
\end{Theorem}
\begin{proof}
	First note that by Proposition~\ref{thm:conecoveringforastate} and Lemma~\ref{lem:powerofnmeanssizeone}, if $n$ is not prime, then for a prime divisor $p$ of $n$, $T(\sfrac{n}{p}, p)$ has infinite order; if $n = p^{l}$ for a prime number $l$, then $T(\sfrac{n}{p},p)$ has order $l$.
	
	Let $T \in \TLn{n}$ be arbitrary. If $T$ is not orientation preserving, then $ \T{R}T$ is. Thus suppose that $T$ is orientation preserving. Let $p_1 \ldots p_r$ be prime numbers and $l_1, \ldots, l_r \in \N_{1}$ be natural numbers such that $s(T)p_1^{l-1} \ldots p_r^{l_r} = n^{m}$ for some $m \in \N$. Then, by Proposition~\ref{thm:conecoveringforastate}, $$U:=T T(\sfrac{n}{p_1},p_1 )^{l_1}\ldots T(\sfrac{n}{p_r},p_r )^{l_r},$$ satisfies $s(U) = 1$. Thus Lemma~\ref{lem:powerofnmeanssizeone} implies that $U = \id$.
	
	Hence we conclude that
	 $$\gen{\T{R}, T(\sfrac{n}{p},p) : p \mathrm{ \ is \ a \ prime \ dividing \ }n} = \TLn{n}.$$
	 
	 Note that for $p $ a prime divisor of $n$, $s(\R T(\sfrac{n}{p}, p)\R) = s(T(\sfrac{n}{p}, p))$. This means, by the above arguments, that, $(\R T(\sfrac{n}{p},p) \R)^{-1} = T(\sfrac{n}{p},p)^{-1}$. Therefore $\R T(\sfrac{n}{p},p) \R = T(\sfrac{n}{p},p)$ and $T(\sfrac{n}{p},p) \R  = \R T(\sfrac{n}{p},p)$.
	 
	 If $n = p^{l}$ for $p$ a prime,  then it is clear, from above, that the orientation preserving subgroup of $\TLn{n}$ is equal to $\gen{T(\sfrac{n}{p}, p)} \cong C_{l}$. Thus, we see that when $n = p^{l}$, $\TLn{n} \cong C_2 \times C_{l}$.
	 
	 Now suppose that $n = p_1^{l_1} \ldots p_{r}^{l_r}$ for $r >1$ and $p_i$ the distinct prime divisors of $n$. Let $u,v$ be distinct prime divisors of $n$. Then,  as before, we note that $s( T(\sfrac{n}{u},u) T(\sfrac{n}{v},v)T(\sfrac{n}{u},u)^{-1} ) = s(T(\sfrac{n}{v},v))$ since $s(T(\sfrac{n}{u},u))s(T(\sfrac{n}{u},u)^{-1})$ is a power of $n$. In particular, we again conclude, by Lemma~\ref{lem:powerofnmeanssizeone}, that $(T(\sfrac{n}{u},u) T(\sfrac{n}{v},v)T(\sfrac{n}{u},u)^{-1})^{-1} = T(\sfrac{n}{v},v)^{-1}$ and so $T(\sfrac{n}{u},u) T(\sfrac{n}{v},v) = T(\sfrac{n}{v},v)T(\sfrac{n}{u},u)$. Thus the generators of $\TLn{n}$ commute.
	 
	 We therefore see that the restriction of the map $s: \Ln{n} \to \Gn{n}$ to the subgroup $\TLn{n}$ yields an isomorphism. 
\end{proof}

We note that Example~\ref{Example:Tnuntodivisorsofn} means that the map $\rsig: \TOn{n} \to \Un{n}$ is unto the subgroup $\Un{n-1,n}$ generated by the divisors of $n$. However, it remains open whether $\Un{n-1,n}$ is precisely the image of the map $\rsig$. 

\begin{comment}
\section{A finite index subgroup of \texorpdfstring{$\mathcal{X}_{n}$}{Lg} }
Throughout this section $\mathcal{X}$ will denote an element of the set $\{ \mathcal{O}, \mathcal{TO}\}$. We are interested in the subgroup $\rsig^{-1}(\Un{n-1,n})$ of $\mathcal{X}_{n}$ and we denote this group by $\mathcal{X}_{\Omega(n)}$. The aim  of this section is to extend results in \cite{BelkBleakCameronOlukoya3} for the group $\Ln{n}$ to the subgroup $\XOno{n}$ of $\XOn$.

We begin with the following lemma.

\begin{lemma}
	Let $T \in  \mathcal{X}_{\Omega(n)}$. Then there is a minimal number $s(T) \in \N$ a divisor of a power of $n$ not divisible by $n$ such that for any state $q \in Q_{T}$, $m_q \cong s(T)$.
\end{lemma}
\end{comment}

\section{An embedding Thompson's group \texorpdfstring{$F$}{Lg} into \texorpdfstring{$\mathcal{O}_{2}$}{Lg}}\label{Section:embeddingF}
The article \cite{MBrinFGuzman} shows that Thompson group $F$ embeds in the group of outer automorphisms of generalisations $F_{n}$ of $F$,  and the $T_{n,n-1 }$, $n>2$. The article \cite{OlukoyaAutTnr} extends this result to show that $F$ embeds into the groups $\TOnr$ and  $\Onr$ of outer automorphisms of $\Tnr$ and $G_{n,r}$ respectively whenever $n \ge 3$.  It is impossible that $F$ embeds into $\TOms{2}{1}$ the outer automorphism group of $T$ since $\TOms{2}{1}$ is finite. However, the group  $\Oms{2}{1}$ is infinite and so might possibly contain an isomorphic copy of $F$. In  this section we show that this is in fact the case. The embedding we give arises from the marker constructions which have been used to great effect in the literature surrounding the automorphisms of the shift dynamical system (\cite{Hedlund,BoyleKrieger, BoyleLindRudolph88, KimRoush, VilleS18}). In particular we show that for any $n \ge 3$, $\mathcal{O}_{2}$ contains an isomorphic copy of a subgroup  $\On^{x} \le  \Oms{n}{1}$ which contains a copy of $F$ .

Throughout we fix $n \ge 3$ and $1 \le x \le n-1$. We consider the subset  $\On^x$ of $\On$ consisting of all elements  $T$ which satisfy the following conditions:
\begin{enumerate}[label=X.\arabic*]
	\item there is a  (necessarily) unique state $q_x \in Q_{T}$ such that $\pi_{T}(x, q) = q_x$ for all $q \in Q_{T}$; \label{Onx:condloop}
	\item for any state $q \in Q_{T}$, $\lambda_{T}(x,q) = wx$ where $w \in \xns$ does not contain $x$; \label{Onx:condxread}
	\item for any word $u \in \xns$ and any state $q \in Q_{T}$ such that $w=\lambda_{T}(u,q)$ contains $x$, then there is a minimal prefix $u_1$ of $u$ such that $u_1$ does not contain $x$, $\lambda_{T}(u_1,q)$ does not contain $x$ and $x$ is maximal suffix of  $\lambda_{T}(u_1x, q)$ in $\{x\}^{\ast}$. \label{Onx:condcontainsx} 
\end{enumerate}

\begin{lemma}
	The subset $\On^{x}$ is a subgroup of $\Ons{1}$
\end{lemma}
\begin{proof}
	First observe that for $T \in \On^{x}$, the state $q_{x}$ must satisfy, by conditions~\ref{Onx:condxread} and \ref{Onx:condcontainsx}, $\lambda_{T}(x, q_x) = x$. 
	
	Checking that $\On^{x}$ is closed under taking products follows from direct computation applying conditions~\ref{Onx:condloop} to \ref{Onx:condcontainsx}.
	
	It now remains to check that $\On^{x}$ is closed under taking inverses.
	
	Let $T'$ be the inverse of $T$ as in construction~\ref{construction:inverse}. Note that there is a minimal $i \in \N_{1}$ such that $U_{x^{i}} \subseteq \im(q_x)$. However, by condition~\ref{Onx:condcontainsx}, we must have that $L_{q_x}(x^{i}) = x^{i}$. For suppose there is a word $u \in \xns$ which does not begin with $x$, such that $\lambda_{T}(u, q_x)$ has a prefix $x$. Then, condition~\ref{Onx:condcontainsx} implies that there is a prefix $u_1$ of $u$ such that $\lambda_{T}(u_1, q_x) = \ew$, $\lambda_{T}(u_1x, q_x) = x$ and $\pi_{T}(u_1x, q_x) = q_x$. However, this now contradicts the injectivity of $q_x$. Therefore, it must that $q_x$ is a homeomorphism state.
	
	Thus, every element of $\On^{x}$ possesses a homeomorphism state and so $\On^{x} \subseteq \Ons{1}$.
	
	Note that for any word $u \in \xns$ and any state $q \in Q_{T}$ such that $U_{u} \subseteq \im(q)$, if $u$ contains $x$, then as $q_x$ is a homeomorphism state $L_{q}(u)$ cannot be empty as there must be a prefix $w$ of $L_{q}(u)$ such that $\pi_{T}(w, q) = q_x$ by condition~\ref{Onx:condcontainsx}. It therefore follows that  any pair $(u,q) \in Q_{T'}$ is such that $u$ does not contain $x$.	
	
	Let $(u,q)$ be any state of $T'$ and $v \in \xnp$ be such that $w=L_{q}(uv)$ contains $x$. Let $w_1$ be the minimal prefix of $w$ such that $w_1 x$ is a prefix of $w$. Let $v_1$ be minimal such that $L_{q}(uv_1) $ has $w_1 x$ as a prefix. Then we see that $\lambda_{T}(w_1x, q) = \nu x$ is a prefix of $uv_1$. Moreover, minimality of $w_1$ and the fact that $p_x$ is a homeomorphism state, implies that $\lambda_{T}(w_1, q)$ does not contain $x$. Therefore, by condition~\ref{Onx:condxread}, $\nu$ does not contain $x$. However, we now have  that $L_{q}(\nu x) = w_1 x$ and so $v_1 = w_1 x$. Thus $ v= v_1 v_2$ where $\lambda_{T'}(v_1, q)$, a prefix of $w_1$ does not contain $x$, and $x$ is the maximal suffix of $\lambda_{T'}(v_1x, q)$ in $\{x\}^{\ast}$.
	
   We note that since $q_x$ is a homeomorphism state, $L_{q}(ux)$ cannot be empty. Moreover, setting $\mu = L_{q}(ux)$, we must have $\pi_{T}(\mu, q) =  q_x$ and $\lambda_{T}(\mu, q) = ux$ as $q_x$ is a homeomorphism state. Condition~\ref{Onx:condcontainsx} now implies that $\mu = \mu_1 x$ where $\mu_1$ does not contain $x$ and $\lambda_{T}(\mu, q)$ does not contain $x$. Therefore we see that $\lambda_{T'}(x, (u,q)) = \mu_1 x$ and $\pi_{T'}(x, (u,q)) = (\ew, q_x)$.
\end{proof}

We show that for any $m \ge 2$, $\Oms{m}{1}$ contains a subgroup isomorphic to $\On^{x}$ using the marker construction.

We need some terminology.

\begin{Definition}
	Let $w_1, w_2 \in \xnp$, then $w_1$ and $w_2$ are said to \textit{overlap non-trivially}, if no proper  suffix of $w_1$ coincides with a proper prefix of $w_2$ and no proper suffix of $w_2$ coincides with a proper prefix of $w_1$. A word $w \in \xns$ is said to \textit{overlap itself trivially} if it does not overlap with itself non-trivially. A subset $w \subset \xns$ is said to have only \textit{trivial overlaps} if no two elements of $w$ overlap non-trivially.
\end{Definition}

The following lemma is essentially a result in combinatorics on words. A proof in a more general context can be found in \cite{BoyleLindRudolph88}.

\begin{lemma}\label{lem:markergeneration}
	Let $n \in \N_{2}$, then there is a collection $B = \cup_{i \in \N_{1}} B_{i} \subseteq \xns$ such that $B_i$ contains $i$ words of equal length and $B_i$ has only trivial overlaps.
\end{lemma}

Let $T \in \On^{x}$ we construct an element $f_T \in \Oms{m}{1}$ as follows. We define $f_T$ as the core of a strongly synchronizng rational self-homeomorphism of $X_{m}^{\omega}$.

Let $B = \{b_0, b_1, \ldots, b_{n-1} \} \subset X_{m}^\ast$ consists of  words of equal length such that $B$ has only trivial overlaps. Let $\iota: B \to \xn$ be the map $b_{i} \mapsto i$. The map $\iota$ extends naturally to a monoid isomorphism $\iota: B \to \xns$ with inverse  denoted $\iota^{-1}$.

 We define a function $f_{T}: X_{m}^{\omega} \to X_{m}^{\omega}$ by first defining it on a dense subset of $X_{m}^{\omega}$ and then taking the unique continuous extension to all of $X_{m}^{\omega}$.
 
  Let $x = x_0x_1 \ldots \in X_{m}^\omega$ be such that $x$ does not have an infinite suffix in $B^{\ast}$. Let $\{ (s_i, t_i)\} \subset \N \times \N$ be defined as follows. Firstly, $s_1$  is minimal and $t_1$ maximal  such that $x_{s_1} \ldots x_{t_1} \in B^{+}$, $t_1 < s_2$ is minimal and $s_2\le  t_2$ maximal such that $x_{s_2} \ldots x_{t_2} \in B^{+}$; inductively, let $t_{i}< s_{i+1}$ be minimal and $s_{i+1} \le t_{i+1}$ maximal, if they exists, such that $x_{s_{i+1}} \ldots x_{t_{i+1}} \in B^{\ast}$. We note that the fact that $B$ has only trivial overlaps is means that an element of $B^{+}$ which occurs as a subword of $x$ can only occur as a subword of $x_{s_{i}} \ldots x_{t_i}$ for some $i  \in \N_{1}$.
  
   Set $w_{0} = x_0 \ldots x_{s_1 -1}$, if $s_{1} = 0$, then $w_0 = \ew$, and for $j \in \N_{1}$ set $w_{j} = x_{t_{j}+1} \ldots x_{s_{j+1}-1}$. For $i \in \N$, let $v_i = (x_{s_{i}}\ldots x_{t_i})\iota$. 
  
  Define $u_i$ as follows: if $v_i$ ends in $x$, then $u_i =  (\lambda_{T}(v_i, q_{x}))\iota^{-1}$; if $v_i$ does not end in $x$, then $\lambda_{T}(v_i x, q_x) = (u_i)\iota^{-1}x$, where we note that $b_{x}$ is not a suffix of  $u_i$.  Let $y \in X_{m}^{\omega}$ be defined as follows $y = w_0 u_1 w_1 u_2 w_2 u_2 \ldots$. Then set $(x)f_{T} :=y$.  We note that the the fact that $B$ only has trivial overlaps means that if $\{(s'_i, t_{i'})\}$ are defined for $y$ analogously as for $x$, then, $s'_1 = |w_0|$, $t'_1 = |u_1|$, and, inductively, $s'_{i+1} = t'_{i} + |w_{i}|$ and $t'_{i+1} = s'_{i+1} + |u_{i+1}|$ .
  
  We note that $f_{T}$ is continuous on the dense subset of $X_{m}^{\omega}$ consisting of all elements which do not contain a right infinite suffix in $B^{\ast}$. Thus $f_{T}$ can be extended uniquely to $X_{m}^{\omega}$.
 
Let $U$ be the inverse of $T$ in $\On^{x}$. We show that $f_{T} f_{U} = \id$. This follows from the following key observations. First observe that the state $p_x$ of $U$ satisfying $\pi_{U}(x, p_x) = p_x$ is such that $U_{p_x}T_{q_x} = \id$. Let $w \in \xnp$ be a word that does not end in $x$ and let $u \in \xnp$ be defined such that $\lambda_{T}(wx, q_x) = ux$. Notice that $u$ must necessarily not end in $x$ by Condition~\ref{Onx:condcontainsx}. Now, since $q_x$ and $p_x$  are homeomorphism states, and  $U_{p_x}T_{q_x} = \id$, we have $\lambda_{U}(ux, p_{x}) = wx$. In the case where $w \in \xnp$ is a word that ends in $x$, then, $\lambda_{U}(\lambda_{T}(w, q_x), p_x) = w$. The last sentence of the previous paragraph now implies that $f_{T}f_{U}  = \id$ on a dense subset of $X_{m}^{\omega}$. Continuity now implies that $f_{T}f_{U} = \id$.  Notice also that since $T$ is strongly synchronizing and minimal, $T_{q_x}$ is the identity map if and only if $T$ is the single-state identity transducer. Thus we see that $f_{T}$ is the identity map if and only if $T$ is.

To conclude that $\On^{x}$ embers in $\Oms{m}{1}$, it suffices to prove two things:
\begin{itemize}
	\item for $U, T \in \On^{x}$, $f_{T}f_{U} = f_{TU}$;
	\item $f_{T}$ is induced by a core and strongly synchronizing transducer.
\end{itemize}

The first follows straightforwardly since if $(qp)_{x}$, $q_x$ and $p_{x}$ are the unique states of $TU$, $T$ and $U$ such that, for $(D,d) \in \{ (TU, qp), (T,q), (U,p) \}$, $\pi_{D}(x, d_{x}) = d_x$, then $(TU)_{(qp)_{x}} = T_{q_x}U_{p_x}$. The result now follows since $f_{D}$ is uniquely determined by the action of the state $D_{x}$.

For the second point we begin with the following observations. Let $L$ be the size of the common length of elements of $B$. Let $T \in \On^{x}$ and let $k$ be the unique synchronizing level of $T$. Let $w$ be a word of length $kL$. Now one of the following alternatives holds:
\begin{enumerate}[label = (\alph*)]
	\item $w$ is an element of $B^{k}$, \label{Cond:normal}
	\item $w$ has a prefix of the form $uv_1 w_1 v_2u$ where $u$ is a suffix of an element of $B$, $v_1 v_2 \in B^{+}$ and $w_1$ has no element of $B$ as a prefix or suffix; \label{cond:abfront}
	\item $w$ has a suffix of the form $v_1 w_1 v_2 u$ where $u$  is a prefix of an element of $B$, $v_1, v_2 \in B^{+}$ and $w_1$ has no element of $B$ as a prefix or suffix.\label{cond:abback}
\end{enumerate}

We consider alternatives~\ref{cond:abfront} and \ref{cond:abback} first. Suppose that $w$ has a prefix of the form $uv_1w_1v_2$ as described in \ref{cond:abfront}. Let $l =uv_1w_1v_2$ and  let $x \in X_{m}^{\omega}$ be any element. Let $i \in \N$ be arbitrary such that $x_i \ldots x_{i+l-1} = uv_1w_1v_2$. We note that since $B$ has only trivial overlaps, the natural numbers $s_j \le t_j$ such that $x_{s_j} \ldots x_{t_j} \in B^{\ast}$ and $\{ s_{j}, s_{j}+1, \ldots, t_{j}\}$contains the index $i + |u| -1$ must satisfy, $t_{j} = i + |uv_1|-1$. Thus, the trivial overlap condition, forces that the numbers $s_{j+1} \le t_{j+1}$ must also satisfy $s_{j+1} = i+ |uv_1w_1|-1$ and $t_j \ge  i+l -1$. In particular the action of $f_{T}$ on the suffix $v_2x_{i+l}x_{i+l+1} \ldots$ is uniquely determined, by definition, by $q_{x}$. Thus we see that the word $uv_1w_1v_2$ forces a unique local action of $f_{T}$. Thus the word $w$ forces a unique local action of $f_{T}$.

  A very similar arguments  shows that if $w$ has a suffix $v_1 w_1 v_2 u$ as described in alternative \ref{cond:abback}, then $w$ again prescribes a unique local action of $f_{T}$ determined by the suffix $v_1 w_1 v_2 u$.

 Now suppose that  $w$ is an element of $B^{k}$. Let  $v$ be the element of $\xn^{k}$ given by $(w)\iota$ and let $q_{v}$ be the state of $T$ forced by  $v$. Let $x \in X_{m}^{\omega}$ be arbitrary. Let $i \in \N$ be arbitrary such that $x_{i}\ldots x_{i+kL -1} =w$. Let $s_j < t_j$ be the natural numbers, such that  $x_{s_j} \ldots x_{t_j} \in B^{+}$ and $\{ s_j, s_{j}+1, \ldots, t_j \}$ contains $i$. Then, by the trivial overlap condition for $B$, $t_{j} \ge i+kL-1$ and $i = s_j + ak$ for some $a \in \N$. In particular, the action of $f_{T}$ on the suffix $x_{i+kL}x_{i+kL+1} \ldots$ is uniquely determined $q_{v}$. Thus we see that the word $w$ determines a unique local action of $f_{T}$.
 
 Thus we see that $f_{T}$ is strongly synchronizing at level  at most $kL$. To see that $f_{T}$ is core, we observe that the local action induced by the word $b_{x}$ is precisely the action of $f_{T}$.
 
 We have thus proved the following result.
 
 \begin{Theorem}
 	Let $m, n \in \N_{2}$ then the group $\Oms{m}{1}$ contains an isomorphic copy of the group $\On^{x}$ for any $x \in \xn$.
 \end{Theorem}

We note that if $T \in \On^{x} \backslash \Ln{n}$, then it is not hard to verify that $f_{T}$ is an element of $\Oms{m}{1}\backslash \Ln{m}$. If however $T \in  \On^{x} \cap \Ln{n}$, then $f_{T} \in \Ln{m}$.

We now prove the main result of this section. To do this it suffices to show that $\On^{x}$ for  some $n \in \N_{2}$ and some $x \in \xn$ contains an isomorphic copy of Thompson's group $F$. 

Let $n \in \N_{3}$ and let $x$ represent an arbitrary element of $X_n \backslash\{0, n-1\}$. Consider the transducers below:

\begin{figure}[H]
	\centering
	\begin{tikzpicture}[shorten >= .5pt,on grid,auto] 
	\node[state] (q_0)   {$p$};
	\node[state, xshift=4cm] (q_1) {$q$}; 
	\node[state, xshift=2cm, yshift=-4.3cm] (q_2) {$s$};
	\node[state, xshift=2cm, yshift=-2cm] (q_3) {$t$};
	\path[->] 
	(q_0) edge[in=170, out=10] node {$0|0$}  (q_1)
	edge[in=105, out=75,loop] node[swap] {$x|x$} ()
	edge[in=195, out=185] node[swap] {$n-1|n-1$} (q_2) 
	(q_1) edge[in=350, out=190] node {$x|n-1x$}(q_0)
	edge[in=10, out=260] node[swap] {$0|\ew$}(q_3)
	edge[in=0, out=0] node{$n-1|(n-1)^2$}(q_2)
	(q_2) edge[in=200, out=180] node[swap] {$x|x$}(q_0)
	edge[in=270, out=240,loop] node[swap, yshift=0.1cm] {$0|0$}()
	edge[in=270, out=300,loop] node[yshift=0.1cm] {$n-1|n-1$}()
	(q_3) edge[in=275, out=170] node[swap] {$x|x$} (q_0)
	edge node[swap] {$n-1|n-10$} node {$0|0$} (q_2);
	\end{tikzpicture}
	\caption{An element  $B \in \T{TO}_{n,1}$. }
	\label{Figure:ThetransducerBinTOn1}
\end{figure}
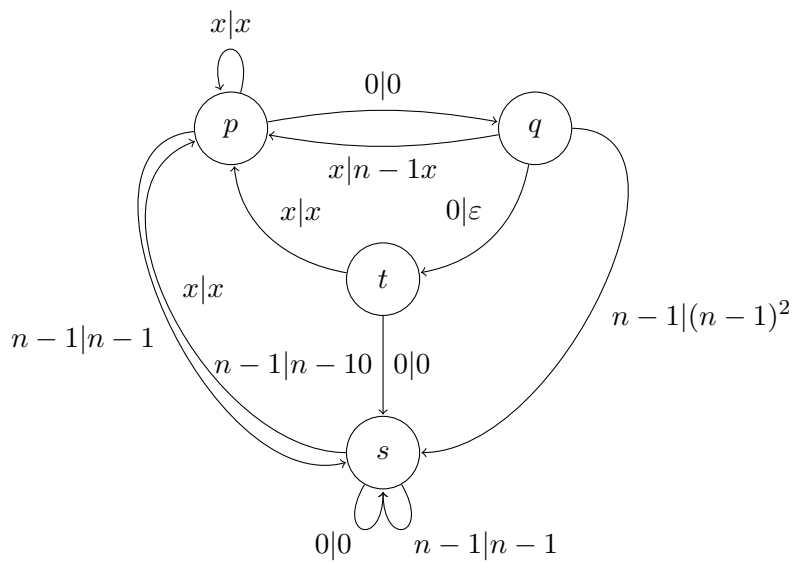

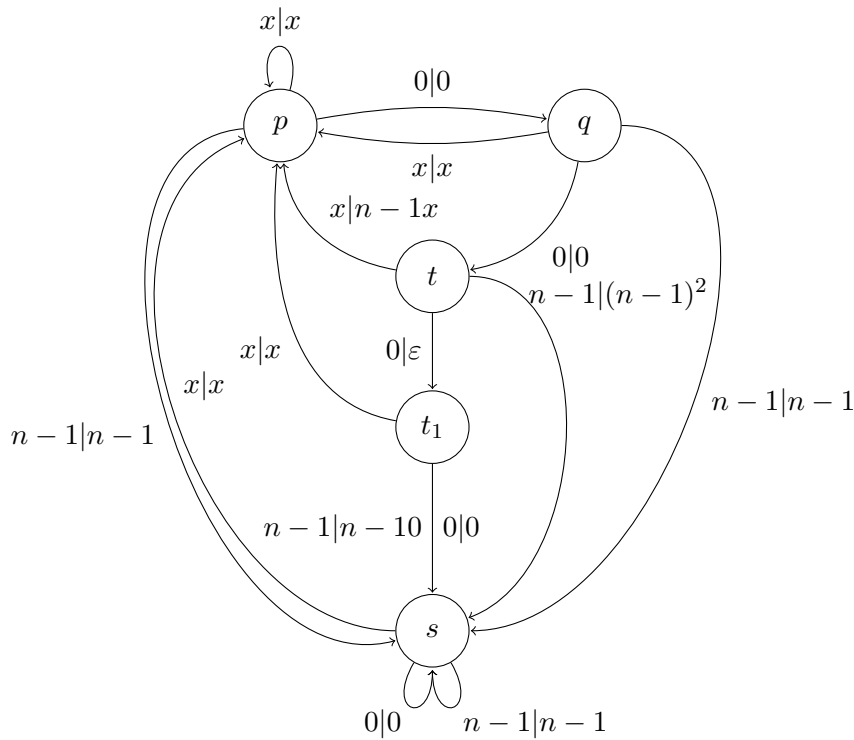
\begin{figure}[H]
	\centering
	\begin{tikzpicture}[shorten >= .5pt,on grid,auto] 
	\node[state] (q_0)   {$p$};
	\node[state, xshift=4cm] (q_1) {$q$}; 
	\node[state, xshift=2cm, yshift=-2cm] (q_2) {$t$}; 
	\node[state, xshift=2cm, yshift=-6.7cm] (q_3) {$s$};
	\node[state, xshift=2cm, yshift=-4cm] (q_4) {$t_1$};
	\path[->] 
	(q_0) edge[in=170, out=10] node {$0|0$}  (q_1)
	edge[in=105, out=75,loop] node[swap] {$x|x$} ()
	edge[in=195, out=185] node[swap] {$n-1|n-1$} (q_3) 
	(q_1) edge[in=350, out=190] node {$x|x$}(q_0)
	edge[in=10, out=260] node {$0|0$}(q_2)
	edge[in=0, out=0] node{$n-1|n-1$}(q_3)
	(q_2) edge[in=275, out=170] node[swap] {$x|n-1x$} (q_0)
	edge node[swap] {$0|\ew$}  (q_4)
	edge[in=20, out=0] node[xshift=2cm, yshift=1.8cm,swap] {$n-1|(n-1)^2$}  (q_3)
	(q_3) edge[in=200, out=180] node[swap] {$x|x$}(q_0)
	edge[in=270, out=240,loop] node[swap, yshift=0.1cm] {$0|0$}()
	edge[in=270, out=300,loop] node[yshift=0.1cm] {$n-1|n-1$}()
	(q_4) edge[in=265, out=170] node {$x|x$} (q_0)
	edge node[swap] {$n-1|n-10$} node {$0|0$} (q_3);
	\end{tikzpicture}
	\caption{An element  $C \in \T{TO}_{n,1}$. }
	\label{Figure:ThetransducerCinTOn1}
\end{figure}

We note that both $B$ and $C$ are elements of $\TOns{1} \cap \On^{x}$ as they satisfy conditions ~\ref{Onx:condloop} to \ref{Onx:condcontainsx}. Moreover, the paper \cite{OlukoyaAutTnr} shows that the restrictions $b, c$ of the elements $B_{p}, C_{p}$ respectively to the subspace  $\{0, n-1\}^{\omega}$ of $\xn^{\omega}$ gives an isomorphism from the subgroup $\gen{B, C}$ of  $\TOns{1} \cap \On^{x}$ to the group of homeomorphisms of $\{0,1\}^{\omega}$ generated by $b$ and $c$. It is not hard to verify, that the group $\gen{b, c}$ is in fact isomorphic to the copy of Thompson's group $F$ acting on the interval $[0, 1/2]$. 

The above fact, together with the containments $\Ons{1} \le  \Onr$ $1 \le r \le n-1$, yield the following corollary:

\begin{corollary}
	Let $n \in \N_{2}$, then $\Onr$ contains an isomorphic copy of Thompson's group $F$. 
\end{corollary}

\def\cprime{$'$}
\providecommand{\bysame}{\leavevmode\hbox to3em{\hrulefill}\thinspace}
\providecommand{\MR}{\relax\ifhmode\unskip\space\fi MR }
% \MRhref is called by the amsart/book/proc definition of \MR.
\providecommand{\MRhref}[2]{%
  \href{http://www.ams.org/mathscinet-getitem?mr=#1}{#2}
}
\providecommand{\href}[2]{#2}


\begin{thebibliography}{10}

\bibitem{ACS}
Jo\~{a}o Ara\'{u}jo, Peter~J. Cameron, and Benjamin Steinberg, \emph{Between
  primitive and 2-transitive: synchronization and its friends}, EMS Surv. Math.
  Sci. \textbf{4} (2017), no.~2, 101--184. \MR{3725240}

\bibitem{BelkBleakCameronOlukoya3}
Jim Belk, Collin Bleak, Peter Cameron, and Feyishayo Olukoya,
  \emph{Automorphisms of the higman thompson groups and the shift dynamical
  system: extensions}, In Preparation, 2019, pp.~1--32.

\bibitem{BlkYMaisANav}
Collin Bleak, Peter Cameron, Yonah Maissel, Andr{\'e}s Navas, and Feyishayo
  Olukoya, \emph{The automorphism groups of the higman-thompson family
  {$G_{n,r}$}}, Submitted, 2016.

\bibitem{BleakCameronOlukoya1}
Collin Bleak, Peter Cameron, and Feyishayo Olukoya, \emph{Automorphisms of the
  higman thompson groups and the shift dynamical system: one-sided case}, In
  Preparation, 2019, pp.~1--41.

\bibitem{BleakCameronOlukoya2}
\bysame, \emph{Automorphisms of the higman thompson groups and the shift
  dynamical system: two-sided case}, In Preparation, 2019, pp.~1--48.

\bibitem{BoyleKrieger}
Mike Boyle and Wolfgang Krieger, \emph{Periodic points and automorphisms of the
  shift}, Trans. Amer. Math. Soc. \textbf{302} (1987), no.~1, 125--149.
  \MR{887501}

\bibitem{BoyleLindRudolph88}
Mike Boyle, Douglas Lind, and Daniel Rudolph, \emph{The automorphism group of a
  shift of finite type}, Trans. Amer. Math. Soc. \textbf{306} (1988), no.~1,
  71--114. \MR{927684 (89m:54051)}

\bibitem{MBrin2}
Matthew~G. Brin, \emph{The chameleon groups of {R}ichard {J}. {T}hompson:
  automorphisms and dynamics}, Inst. Hautes \'Etudes Sci. Publ. Math. (1996),
  no.~84, 5--33 (1997). \MR{1441005}

\bibitem{MBrinFGuzman}
Matthew~G. Brin and Fernando Guzm{\'a}n, \emph{Automorphisms of generalized
  {T}hompson groups}, J. Algebra \textbf{203} (1998), no.~1, 285--348.
  \MR{1620674}

\bibitem{GriNekSus}
R.~I. Grigorchuk, V.~V. Nekrashevich, and V.~I. Sushchanski{\u\i},
  \emph{Automata, dynamical systems, and groups}, Tr. Mat. Inst. Steklova
  \textbf{231} (2000), no.~Din. Sist., Avtom. i Beskon. Gruppy, 134--214.
  \MR{1841755 (2002m:37016)}

\bibitem{Hedlund}
G.~A. Hedlund, \emph{Endomorphisms and automorphisms of the shift dynamical
  system}, Math. Systems Theory \textbf{3} (1969), 320--375. \MR{0259881}

\bibitem{KimRoush}
K.~H. Kim and F.~W. Roush, \emph{On the automorphism groups of subshifts}, Pure
  Math. Appl. Ser. B \textbf{1} (1990), no.~4, 203--230 (1991). \MR{1137698}

\bibitem{OlukoyaAutTnr}
Feyishayo Olukoya, \emph{Automorphisms of the generalised {T}hompson's group
  {$T_{n,r}$}}, In Preparation, 2018, pp.~1--35.

\bibitem{Ryan2}
J.~Patrick Ryan, \emph{The shift and commutivity. {II}}, Math. Systems Theory
  \textbf{8} (1974/75), no.~3, 249--250. \MR{0383384}

\bibitem{VilleS18}
Ville Salo, \emph{A note on subgroups of automorphism groups of full shifts},
  Ergodic Theory Dynam. Systems \textbf{38} (2018), no.~4, 1588--1600.
  \MR{3789178}

\bibitem{Volkov2008}
Mikhail~V. Volkov, \emph{Language and automata theory and applications},
  Springer-Verlag, Berlin, Heidelberg, 2008, pp.~11--27.

\end{thebibliography}
\end{document}